\documentclass[11pt,reqno]{amsart}
\usepackage{amssymb, amsmath, amsthm, graphicx, color, epsfig, amsfonts,bbm}

\newtheorem{thm}{Theorem}[section]
\newtheorem{lem}[thm]{Lemma}
\newtheorem{cor}[thm]{Corollary}
\newtheorem{prop}[thm]{Proposition}

\newtheorem{rem}{Remark}[section]

\newtheorem{optimization}{Optimization}
\numberwithin{equation}{section}

\newcommand{\K}{{\mathcal K}}
\newcommand{\Q}{{\mathcal Q}}
\renewcommand{\P}{{\mathcal P}}

\def\R{{\mathbb{R}}}
\def\N{{\mathbb{N}}}
\def\Z{{\mathbb{Z}}}

\def\E{{\mathcal{E}}}
\def\B{{\mathcal{B}}}
\def\O{{\mathcal{O}}}
\newcommand{\I}{{\mathcal I}}
\newcommand{\z}{R}

\newcommand{\one}{{\mathbbm{1}}}
\allowdisplaybreaks

\title{Atypical behaviors of a tagged particle in asymmetric simple exclusion}

\date{\today } 

\begin{document}

\begin{abstract} 
Consider the asymmetric nearest-neighbor exclusion process (ASEP) on $\Z$ with single particle drift $\gamma>0$, starting from a Bernoulli product invariant measure $\nu_\rho$ with density $\rho$.  It is known that the position $X_{N}$ of a tagged particle, say initially at the origin, at time $N$ satisfies an a.s. law of large numbers $\frac{1}{N}X_N \rightarrow \gamma(1-\rho)$ as $N\uparrow\infty$.

In this context, we study the `typical' behavior of the tagged particle and `bulk' density evolution subject to `atypical' events $\{X_N\geq AN\}$ or $\{X_N\leq AN\}$ for $A\neq \gamma(1-\rho)$.  We detail different structures, depending on whether $A<0$, $0\leq A< \gamma(1-\rho)$, $\gamma(1-\rho)<A< \gamma$, or $A\geq \gamma$, under which these atypical events are achieved, and compute associated large deviation costs.  Among our results is an `upper tail' large deviation principle in scale $N$ for $\frac{1}{N}X_N$.  

\end{abstract}


 
 \author{Sunder Sethuraman}
\address{Department of Mathematics\\
University of Arizona\\
621 N. Santa Rita Ave.\\
Tucson, AZ 85750, USA}
\email{{\tt sethuram@math.arizona.edu}}

\author{S.R.S. Varadhan}
\address{Courant Institute\\
New York University\\
251 Mercer St.\\
New York, NY, 10012, USA}
\email{{\tt varadhan@cims,nyu.edu}}
 
\maketitle
%\tableofcontents

\section{Introduction}
Informally, the nearest-neighbor exclusion process on $\Z$ follows the unlabeled evolution of a typically infinite collection of particles, where a particle at $x\in \Z$ jumps with rate $p(\pm 1)$ to location $x\pm 1$ if the site $x\pm 1$ is unoccupied, and otherwise remains put with the clock reset.  
More carefully,
let 
$$\eta=\{\eta(x):x\in \Z\}\in \Omega= \{0,1\}^{\Z}$$ 
be a particle configuration on $\Z$ where $\eta(x)$ is the indicator of site $x$ being occupied $(=1)$ or empty $(=0)$.  Denote by $p(1) = p$ and $p(-1)=1-p=q$, where $p\in [0,1]$, the nearest-neighbor (translation-invariant) jump rate.
For local functions $f:\Omega\rightarrow \R$, those depending only on a finite number of variables $\{\eta(x):x\in \Z\}$, define 
\begin{align}
\label{exclusion gen}
\mathcal{L}f(\eta) &= \sum_{x\in \Z} \Big\{p\eta(x)(1-\eta(x+1))\big(f(\eta^{x,x+1})-f(\eta)\big) \\
&\ \ \ \ \ + (1-p)\eta(x)(1-\eta(x-1))\big(f(\eta^{x,x-1})-f(\eta)\big)\Big\}, \nonumber
\end{align}
where $\eta^{x,y}$ is the configuration where values $\eta(x)$ and $\eta(y)$ have been exchanged.

 Then, the exclusion process $\{\eta_t: t\geq 0\}$ is the Markov process on $\Omega$ with generator $\mathcal{L}$.  Such a process has been a natural model in which to study different types of `flows' in several scales. For instance, the particles may represent cars moving in a single lane without passing, or the spacings between particles may be interpreted as a series of queues.

The system is `mass conservative', 
and therefore possesses a family of invariant measures.  In particular, it is known that the product measure $\nu_\rho = \prod_{x\in \Z} \text{Bern}(\rho)$ of Bernoulli factors, where the random variable with distribution $\text{Bern}(\rho)$ takes values $1$ and $0$ with probabilities $\rho$ and $1-\rho$ respectively, is invariant for each $\rho\in [0,1]$.   See \cite{DeMP}, \cite{kl}, \cite{Liggett}, \cite{Spohn} for more discussion and history.

To make a choice, we will consider the model with `right-ward' drift, as results in the `left-ward' drift setting will be analogous.
As is now standard, to fix terminology, when $p=1/2$, the model is called the symmetric simple exclusion process (SSEP).  Whereas, when $1/2<p\leq 1$, we will say it is the asymmetric simple exclusion process (ASEP).  In the case, $p=1$, the model is the totally asymmetric exclusion process (TASEP).  We will define $\gamma$ as the drift $\gamma = 2p-1=p-q\geq 0$.

 In the exclusion system, the motion of a tagged, or distinguished particle interacting with others is of basic interest, and of relevance in applications; see \cite{CS}[Section 1] (and references therein),  \cite{kl}, \cite{Liggett}, \cite{Liggett2}, \cite{Spohn} for more discussion. 
Suppose initially the origin is occupied.  Tagging this particle, let $X_t$ be its position at time $t$.  By itself, it is not in general a Markov process with respect to its own history, due to interactions with the other particles.  Indeed, in the nearest-neighbor exclusion dynamics, the tagged particle is always between the particle to its left and that to its right, never jumping over them. 

The belief though is that it behaves to first order as some type of random walk, homogenized with respect to the `bulk' particle mass hydrodynamics.
Such typical law of large numbers (LLN) behaviors are understood in in ASEP and TASEP \cite{Reza_LLN}, \cite{Saada}, \cite{Seppalainen}, as well as in SSEP \cite{JL}.  Moreover, with respect to large deviations, we note in TASEP `lower tail' limits have been found \cite{Seppalainen}.
However, less is known on how the tagged particle optimally behaves in atypical situations, say in the event it exceeds or is less than the typical position at a given time.

Fix $\rho\in (0,1)$.  The focus of our work will be on the ASEP model, when starting from an invariant state $\nu_\rho(\cdot |\eta(0)=1)$, conditioned so that the tagged particle is at the origin.  From such initial measures, as we discuss in the next subsection, the tagged particle evolution possesses interesting properties. 
In this system, our goal is 
to understand how the bulk particle mass and tagged particle optimally organize and move to achieve deviation events.  In the course of these computations, we establish an Euler scale `upper tail' large deviations principle for the tagged particle, and solve related calculus of variations optimizations.
In a case of the results, for `lower tail' events in TASEP, `nonentropic' solutions of the hydrodynamic equation will play a role.

We comment that the diffusive scale large deviations behavior of the tagged particle in SSEP, different than in ASEP, has been studied in \cite{Derrida}, \cite{Sasamoto}, \cite{SV}.
 We note also large deviations of the related but different statistics of the `current' in ASEP, TASEP and SSEP, starting from types of initial conditions, has been studied for instance in \cite{Chle}, \cite{DPS}, \cite{Das}, \cite{Joh}, \cite{Olla_Tsai}, \cite{Seppalainen}, \cite{SV}.

In the following, we will denote by $\P_\rho$ and $\E_\rho$ the process measure and expectation starting from $\nu_\rho(\cdot|\eta(0)=1)$.  More generally, we denote $\P_\mu$ and $\E_\mu$ as the process measure and associated expectation when starting from distribution $\mu$.  We will also have occasion to denote $E_\mu$ as the expectation under measure $\mu$.

\medskip
\noindent{\bf Plan of the article.}
We give an informal discussion of our results in Section \ref{results section}, facilitated by known {\it exact} and {\it approximate} formulas for the law of the tagged particle $X_t$ in Section \ref{exact_section}, and connections between the `bulk' particle mass dynamics and that of the tagged particle in Section \ref{connections section}.  Results are stated in Section \ref{section results}.
After preliminaries in Section \ref{section preliminaries}, proofs of the lower and upper bound results are given in Sections \ref{LB section} and \ref{UB section}. 
The upper bounds rely on a calculus of variations problem solved in Section \ref{calc var section}.  Proofs of preliminary ingredients are found in Sections \ref{LLN appendix}, \ref{LDP section}, \ref{section superexponential}.  The appendix develops a Hopf-Lax formula, as well as a proof of an exact formula in TASEP, for the reader's convenience.

\subsection{Exact and approximate formulas} 
\label{exact_section}
In TASEP, starting the system under distribution $\nu_\rho(\cdot|\eta(0)=1)$, remarkably, the tagged particle position 
\begin{equation}
\label{exact_formula}
X_t \text{ is a Poisson process with rate } 1-\rho.
\end{equation} 
 This follows from viewing TASEP in terms of a zero-range queuing system; see Appendix \ref{Poisson process section} for a derivation.  

Such an exact corresponding formula in ASEP is not available.  However, interestingly, Poissonian approximation is known (cf. \cite{Ferrari_Fontes}:  When the system starts in $\nu_\rho(\cdot|\eta(0)=1)$, there is a $0<\theta = \theta(\rho, p)<\infty$ such that
\begin{equation}
\label{approximate_formula}
X_{t} = Poi_t +\chi_t
\end{equation}
where $Poi_t$ is a Poisson process with rate $\gamma (1-\rho)$ and $\chi_t$ is a stationary process with finite exponential moment $E[\exp\{\theta |\chi_t|\}]<\infty$.

Of course, from these results, introducing a scaling parameter $N\geq 1$ and fixing the `macroscopic' time $t\geq 0$, laws of large numbers $\lim_{N\rightarrow\infty}X_{Nt}/N =\gamma(1-\rho)t$ a.s., and central limit theorems $(X_{Nt} - \gamma(1-\rho)Nt)/\sqrt{N} \Rightarrow {\mathcal N}(0, \gamma(1-\rho)t)$ are straightforwardly deduced starting from $\nu_\rho(\cdot|\eta(0)=1)$; see also \cite{Saada}, \cite{Kipnis} for other derivations of these limits.

Also, from the exact formula \eqref{exact_formula}, in TASEP, a large deviation principle holds:  $\P_\rho(X_{Nt}/N \in F) \sim \exp\{-NI_{1,t}(F)\}$ where the rate function $I_{1,t}(F) = \inf_{a\in F}I_{1,t}(a)$ and
\begin{equation}
\label{Tasep_rate_function}
I_{1,t}(a) = \left\{\begin{array}{rl}
a\log\left(\frac{a}{(1-\rho)t}\right) -a+(1-\rho)t & \ {\rm for \ } a\geq 0\\
\infty& {\rm \ for \ }a<0. \end{array}\right.
\end{equation}
We will denote in the following $I_1 = I_{1,1}$ when $t=1$.  

 In the general ASEP model, less is known.  We remark that, as $Poi_t$ and $\chi_t$ are not independent, a large deviation rate function cannot be found immediately, although bounds may be derived from the approximate formula \eqref{approximate_formula}.
 
  However, it is known in ASEP that the position $X_t$ has negatively associated stationary increments when starting from $\nu_\rho(\cdot|\eta(0)=1)$; see \cite{Kipnis} for a derivation, where the law of the tagged particle is understood as the current in a certain zero-range particle system.  That is, for increasing functions $f, g$, and $s,t>0$, 
 $$\E_\rho[f(X_{t+s}-X_s)g(X_s)] \leq \E_\rho[f(X_t)]\E_\rho[g(X_s)].$$

From the negative association, an upper bound large deviation principle can be stated.  Consider that $\big\{\log \E_\rho[e^{\lambda X_{s}}]: s>0\big\}$ is a subadditive sequence for each $\lambda \in \R$.  Then, as $N\uparrow\infty$, $\frac{1}{N}\log \E_\rho\big[e^{\lambda Nt}\big]$ converges to the pressure
 $\Lambda_t(\lambda) = \inf_{k\geq 1} \frac{1}{k}\log \E_\rho[e^{\lambda X_{kt}}]$ for each $t\geq 0$. 
 Define, a rate function
 \begin{equation*}
 \hat{I}_{\gamma,t}(a) = \sup_\lambda \{ \lambda a - \Lambda_t(\lambda)\},
 \end{equation*}
for $a\in \R$.  Since $\Lambda_t(\lambda)<\infty$ for $\lambda\in \R$, the scaled positions $\{X_{Nt}/N\}$ are exponentially tight.  Hence, by the Gartner-Ellis theorem, for closed $U\in \R$,   
$$\limsup \frac{1}{N}\log \P_\rho\big(\frac{1}{N}X_{Nt} \in U\big) \leq -\inf_{a\in U} \hat{I}_{\gamma,t}(a).$$ 

In TASEP ($\gamma=1$), of course explicitly $\hat{I}_{1, t} = I_{1, t}$.  Otherwise, more generally in ASEP, $\hat{I}_{\gamma,t}$ is not a priori explicit.  But, the Poissonian approximation gives an upper bound for $\Lambda_t(\lambda)$ in terms of the pressure for the Poisson process with rate $\gamma(1-\rho)$.  From H\"older's inequality, one sees for $r\lambda/(r-1)<\theta$ that
   $\Lambda_t(\lambda) \leq \frac{\gamma(1-\rho)t}{r} \big(e^{r\lambda} -1\big)$,
   which in turn gives an upper bound of $\hat I_{\gamma,t}(a)$ for $a$ close to $\gamma(1-\rho)t$.
We comment that an accompanying lower bound large deviation principle would also be possible if $\Lambda_t(\cdot)$ could be shown to be essentially smooth, which does not seem to follow straightforwardly.

\subsection{Connections between the `bulk' particle mass and the tagged particle}
\label{connections section}
Although the tagged particle has its own `clock' and can move by itself in empty space, it does interact with the other particles.
Indeed, we may formulate a joint Markov process $(\eta_t, X_t)$ on the configuration space $\{0,1\}^\Z \times \Z$ with generator
\begin{align}
\mathcal{L}f(\eta, X) &= \sum_{x\neq X\in \Z} \Big\{p\eta(x)(1-\eta(x+1))\big(f(\eta^{x,x+1}, X)-f(\eta, X)\big) \nonumber\\
&\ \ \ \ \ \ \ \ \ \ \ \ \ \ \ \ + (1-p)\eta(x)(1-\eta(x-1))\big(f(\eta^{x,x-1}, X)-f(\eta, X)\big)\Big\}\nonumber\\
& \ \ \ \ \ + p\eta(X)(1-\eta(X+1))\big(f(\eta^{X,X+1}, X+1)-f(\eta, X)\big)\nonumber \\
&\ \ \ \ \ + (1-p)\eta(X)(1-\eta(X-1))\big(f(\eta^{X,X-1}, X-1)-f(\eta,X)\big).
\label{joint}
\end{align}
  Moreover, for the process $\xi_t(\cdot)=\eta_t(\cdot + X_t)$ in the reference frame of the tagged particle, always at $0$, the measure $\nu_\rho(\cdot |\eta(0)=1)$ is invariant (cf. \cite{Liggett}).

In particular, an ingredient to capture the optimal behaviors of the tagged particle to achieve atypical events will be connections to the `bulk' particle mass evolution:
Define the current $J_t(x)$ across the site $x\in \Z$ as the difference of the number of particles which have crossed $x$ to $x+1$ and those which have crossed from $x+1$ to $x$ up to time $t>0$. 
 Then, the relation
$J_t(y) - J_t(x) = \sum_{z=x+1}^y \left(\eta_0(x) - \eta_t(z)\right)$
 for $x<y$ holds.

 When the initial configuration $\eta_0$ is empty to the left eventually, that is $\eta_0(z)=0$ for all $z<0$ with $|z|$ large, note that
$J_t(0) = \sum_{z\leq 0} \left(\eta_0(z) - \eta_t(z)\right)$.
 In this setting, for the tagged particle, starting from the origin, to move beyond $a\geq 0$, all the particles between it and $a$ initially must have crossed beyond $a$ at time $t$:
$\big\{X_t>a\big\} = \big\{J_t(a)\geq \sum_{z=0}^a \eta_0(z)\big\} =  \big\{\sum_{z\leq a}\eta_t(z) < \sum_{z\leq 0}\eta_0(z)\big\}$.
Similar relations can be written for $a<0$.  See Section \ref{current section} for more discussion.

 Let $T=1$ be a fixed time horizon.  When $\eta_0(x) =0$ for all large $|x|$, denote the scaled empirical measure for $t\in [0,1]$ by
$$\pi^{N}_t = \frac{1}{N}\sum_{x\in \Z}\eta_{Nt}(x)\delta_{x/N},$$
which belongs to the Skorohod space $D([0,1]; M(\R))$, where $M(\R)$ is the space of finite Radon measures, endowed with a distance 
\begin{align}
\label{distance}
d(\mu,\nu) = \sum_{j\geq 1} \frac{1}{2^j} \big |E_\mu [f_j] - E_\nu[f_j] \big |;
\end{align}
 here $\{f_j\}$ is a countable, dense set of smooth, compactly supported functions on $\R$.

Let $\rho_0(u): \R \rightarrow [0,1]$ be a given initial `mass profile', which is piecewise continuous and vanishes at $u$ for all large $|u|$.  Define the local equilibrium measure
$$\nu^N_{\rho_0(\cdot)} = \prod_{x\in \Z} {\rm Bern}\left(\rho_0(x/N)\right).$$
The following hydrodynamic limit is known:  Starting from initial distribution $\nu^N_{\rho_0(\cdot)}$, for compactly supported test functions $J:[0,1]\times \R\rightarrow \R$, as $N\rightarrow\infty$,
$$\int_0^1 \langle J, \pi^{N}_s\rangle ds \stackrel{{\rm prob.}}{\rightarrow} \int_0^1 \int_\R J(s,u)\rho(s,u)duds$$
where
$\rho(t,u)$ is the unique entropic solution of the Burgers type conservation law
\begin{equation}
\label{Burgers}
\partial_t \rho +\gamma\partial_u \{\rho(1-\rho)\} = 0
\end{equation}
for $0<t\leq 1$ such that $\rho(0, \cdot) = \rho_0(\cdot)$.  Moreover, the measure-valued trajectory $\pi^N_\cdot$ in $D([0,1], M(\R))$ converges to $\rho(\cdot, u)du$ in probability; see \cite{Rezakhanlou}, \cite{Rost}, \cite{Seppalainen}.

Suppose $\rho_0$ is positive on both sides of the origin.  Then, starting the process from $\nu^N_{\rho_0(\cdot)}(\cdot|\eta(0)=1)$, the tagged position $X_{Nt}/N$ converges in probability to the unique solution $v_t$ of the ODE
\begin{align}
\label{ODE-reza}
\dot v = \gamma\{ 1-\rho(t,v_t)\},
\end{align}
interpreted in the Fillipov sense, for $t\in [0,1]$; see \cite{Reza_LLN} and also in TASEP \cite{Seppalainen}.   We observe, when $\rho_0\equiv \rho$ is a constant, then \eqref{ODE-reza} gives $v_t=t\left[\gamma(1-\rho)\right]$, already deduced from the Poissonian approximation.  

We also consider in Section \ref{LLN section}, the LLN behavior in ASEP in a couple of scenarios, including starting from a step profile (cf. \cite{TW}), where $\rho_0$ vanishes on one side of the tagged particle, needed in later large deviation arguments.  Here, the motion is not carried by the `bulk' hydrodynamics, and more specific particle arguments are used.

We discuss further that, in TASEP, upper and lower large deviation bounds for the empirical measure have been established.   In particular, consider `nice' evolutions $\rho$:  Let $\{[t_j^-, t_j^+]\}$ be a finite collection of disjoint intervals in $[0,1]$, and $s_j:[t_j^-, t_j^+]\rightarrow \R$ be associated piecewise differentiable functions. We suppose that $\rho$ is continuously differentiable away from the union $\{(t, s_j(t)): t\in [t_j^-, t_j^+]\}$.  Also, assume that $\rho(t, \cdot)$ has right and left limits, 
$t\mapsto \rho(t, s_j(t)\pm)$ is piecewise continuous, and $\rho(t, u) = \rho$ at $(t,u)$ for all large $|u|$.

Then, for such profiles $\rho(\cdot, \cdot)$, when starting from $\nu_\rho$, 
$$\lim_{\delta\downarrow 0}\lim_{N\uparrow\infty}\frac{1}{N}\log \P_\rho\left(\pi^N_\cdot \in B_\delta(\rho)\right) = K(\rho_0)+I_{JV}(\rho),$$
where $B_\delta(\rho(\cdot,u)du)$ is a $\delta$-ball around this profile in the Skorohod space $D([0,1]; M(\R))$.
Here,
\begin{equation*}
K(\rho_0) = \int_\R \left\{\rho_0(u)\log \frac{\rho_0(u)}{\rho} + (1-\rho_0(u))\log \frac{1-\rho_0(u)}{1-\rho}\right\} du
\end{equation*}
and 
\begin{equation*}
I_{JV}(\rho) =  -\inf_\phi \int_0^1 \int_\R \left \{h(\rho)\partial_t\phi + g(\rho)\partial_u\phi \right\} dudt
\end{equation*}
where the infimum is over $\phi\in C^\infty_c((0,1)\times \R)$ such that $0\leq \phi\leq 1$.  Here, $h$ is the convex function $h(u)= u\log(u) + (1-u)\log(1-u)$ and $g(u) = u(1-u)\log(u/(1-u))-u$ is its dual so that $\dot{g} = \dot{f}\dot{h}$ with $f(u) = u(1-u)$. 

The rate $K$ represents the large deviation cost of changing the initial condition, and $I_{JV}$ is the dynamical cost of the profile.
Moreover, the `Jensen-Varadhan' rate $I_{JV}$ is finite only on weak `nonentropic' solutions of \eqref{Burgers}, with initial condition $\rho_0$ (cf. \cite{Jensen_thesis}[Lemmas 2.1, 2.2, and Corollary 2.3]), and vanishes at the entropic solution.  Informally, $I_{JV}$ is the positive charge of `$\big(\partial_t h(\rho) + \partial_x g(\rho)\big)dtdu$'. 
   See \cite{Jensen_thesis}, \cite{Varadhan_Jensen}, \cite{Vilensky}, \cite{Quastel-Tsai} for more details.  

In passing, we note such results in ASEP have not yet been shown.  But, in symmetric exclusion, large deviations in diffusive scale have been carried out in \cite{KOV}; see also \cite{BGL}.

Although we will not use the full form of the $\pi^{N}$-large deviations, an element of the proof for TASEP, that we generalize to ASEP, will be key to capture the large deviation behavior of $X_{Nt}/N$.  Namely, we discuss in Sections \ref {superexponential subsection} and \ref{section superexponential} that the scaled microscopic `height' function of the particles cannot be less than the continuum `height' function with respect to the entropic evolution without incurring superexponential cost.

\subsection{Discussion of main results}
\label{results section}
In this context, we focus our work on the optimal structure of the `bulk' particle mass and tagged particle motion to achieve large deviation events of the type $\{X_{Nt}/N \geq A\}$ or $\{X_{Nt}/N\leq A\}$, where the macroscopic destination $A \neq \gamma(1-\rho)t$ is not equal to the typical level, at a fixed macroscopic time. Without loss of generality, we will take $t=1$.  

 The difficulties in the analysis involve understanding how the bulk should organize with respect to the atypical tagged particle motion.   
 We comment that, since we start from a random initial condition $\nu_\rho(\cdot|\eta(0)=1)$, the natural scale for the large deviations is the `Euler' scale $N$.  Indeed, slowing down the tagged particle only requires the rate of a single particle to be modified up to microscopic time $N$, a large deviation cost of $O(N)$.  However, to speed up the tagged particle motion, if we were to start from a deterministic condition, then typically $O(N)$ particles rates would have to be altered, a cost of $O(N^2)$.  With a random initial condition though, one may simply remove say $O(N)$ particles to give the tagged particle room to manouever, at a smaller large deviation cost of $O(N)$. 

Now, in the setting of TASEP, when $\gamma=1$ and $t=1$, there are three categories of behaviors depending on whether $0\leq A<1-\rho$, $1-\rho<A< 1$, and $A\geq 1$, each with its own typical optimal strategy.  These strategies are such that their large deviation costs match the already known Poisson process rate function \eqref{Tasep_rate_function} when $t=1$.

More generally in ASEP, however, as the tagged particle can go left of the origin, there is another category to consider.  In particular, the categories are $A<0$, $0\leq A<\gamma(1-\rho)$, $\gamma(1-\rho)<A< \gamma$ and $A\geq \gamma$.  We establish here lower large deviation bounds in all regions, some of them though are likely nonoptimal.  Although in ASEP we do not have existence a priori of a large deviation rate function, a suggestion from the Poisson process with mean $\gamma(1-\rho)$ approximation \eqref{approximate_formula} is that for $A$ near the zero, $\gamma(1-\rho)$, the rate function if it exists, would be the same as for the Poisson process.

However, in the general ASEP setting, we show in the two `upper tail' settings, when $A>\gamma(1-\rho)$, that the lower bound costs can be matched by upperbound minimizations with respect to the cost of changing the initial distribution $K$, as well as possibly the cost of changing the rates of the tagged particle;
see Theorems \ref{thm:UB A large} and \ref{thm:UB A not so large}, and the calculus of variations problems in Proposition \ref{Prop:UB} and Section \ref{calc var section} which may be of independent interest.  As a consequence, we may state an `upper tail' large deviations principle in Corollary \ref{thm: upper tail}.

In particular, we establish, independently of the exact and approximate Poisson process formulas for TASEP and ASEP, `upper tail' large deviations for $X_N/N$ in scale $N$ with an explicit rate function $I_\gamma$,
\begin{align}
\label{I_gamma}
I_{\gamma}(A) =\left\{\begin{array}{rl}
A \log(c) - pc - \frac{1-p}{c} +1 - A\log(1-\rho) - \gamma\rho & {\rm \ for \ } A\geq \gamma\\
A\log\left(\frac{A}{\gamma(1-\rho)}\right) - A + \gamma(1-\rho) & {\rm \ for \ }\gamma(1-\rho)<A<\gamma
\end{array}\right.
\end{align}
where
$c = \left(A + \sqrt{A^2 + 4p(1-p)}\right)/(2p)$.
  Naturally, the implicit upperbound rate mentioned before $\hat I_{\gamma, 1}$ satisfies $\hat{I}_{\gamma, 1}(A) = I_\gamma(A)$ for $A>\gamma(1-\rho)$.
	We note in TASEP, when $\gamma=p=1$, the formula \eqref{I_gamma} reduces to the rate function \eqref{Tasep_rate_function} $I_1=I_{1,1}$.  Moreover, here $\hat I_{1,1}=I_1$ for all $A\in \R$.
			
			However, in ASEP, only in the region $\gamma(1-\rho)<A<\gamma$ closer to the zero $\gamma(1-\rho)$ does $I_\gamma(A)$ match a rate function derived from Poisson approximation \eqref{approximate_formula}. In
	the farther regime $A\geq \gamma$, the optimal cost involves changing the $p,q=1-p$ birth-death rates of the tagged particle, which is not a Poisson process with mean parameter $\gamma = p-q$, and as a consequence the rate $I_\gamma(A)$ is smaller than what would be found under Poisson approximation.

Finally, we describe briefly the lower bound strategies involved in each regime.
\vskip .1cm

{\it ASEP: $A\geq \gamma$ (Theorem \ref{thm:ASEP A large}).}
For the tagged particle to move so that $X_N\geq AN$, it must speed up its intrinsic jump rate, and must not be obstructed by other particles.  Given that it is too costly to speed up $O(N)$ particles, 
we must remove particles from the system to achieve the tagged particle deviation.  Indeed, we specify an initial condition for the distribution of particles, where there are no particles between $0$ and $(A-\gamma)N$, and prescribe a change-of-rates for the tagged particle, which taken together has large deviation cost equal to $I_\gamma(A)$.  
       
\vskip .1cm

{\it ASEP: $\gamma(1-\rho)<A< \gamma$ (Theorem \ref{thm:ASEP A not so large}).}
Here, we still remove particles and change the initial profile of the particle distribution, so that near the origin the density is $1-A/\gamma$.  Then, the tagged particle, without changing its rates, will flow at macroscopic rate $A$, faster than the typical macroscopic velocity $\gamma(1-\rho)$.  The associated large deviation cost will match $I_\gamma(A)$.

\vskip .1cm

{\it TASEP: $0\leq A< 1-\rho$ (Theorem \ref{thm:TASEP nonentropy}).}
Consider first the case $A=0$, that is the event that the tagged particle doesn't move macroscopically up to time $t=1$.  One way to achieve this is to suppress the clock so that it doesn't ring up to this time, for which the large deviation cost is $-\log(e^{-1})=1$, strictly larger than $I_1(0)=1-\rho$.  

Another way is to prevent jumps by blocking the motion with other particles.  That is, one can change the initial distribution so that the macroscopic profile is a step function with value $1$ for locations $0\leq u\leq 1$, and value $\rho$ for other $u$.  The corresponding entropic flow involves a a rarefaction wave, whose edge moves back to the origin at time $1$.  
However, the large deviation cost will be $-\log(\rho)> 1-\rho = I_1(0)$ again.  

One may block though with less particles, say with an initial distribution whose macroscopic profile is the step function with value $1$ for $0\leq u\leq \rho$ and value $\rho$ for $u$ otherwise.  We now move the mass {\it nonentropically}, keeping the Riemann step, but moving the shock backwards to the origin with velocity $-\rho$.  Physically, one can interpret that a particle at the front edge  is `slowed' before it moves on.  Alternatively, in terms of particle-hole duality, if we look at the evolution of the `holes', initially $0<u\leq \rho$ is empty, and value $1-\rho$ for $u$ otherwise; then, the velocity of the `hole', initially at $x=\rho N$, is slowed from $-1$ to $-\rho$ up to time $1$.  The large deviation cost $I_{JV}$ of this nonentropic but `nice' solution to \eqref{Burgers}, when added to the $K$ cost changing the initial profile,  
equals $1-\rho = I_1(0)$.

For $0<A<1-\rho$, a related `blocking' scheme and nonentropic evolution, still without changing the rates of the tagged particle, yields also the desired large deviation cost $I_1(A)$.

\vskip .1cm
{\it Other ASEP categories.}
Here, one can follow the same scheme as in TASEP for the category $0\leq A< \gamma(1-\rho)$, however in terms of an entropic evolution,  
and derive a strategy to achieve the event 
$\{X_N \leq A\}$; see Proposition \ref{Prop:ASEP A small}.
  When $A<0$, we give lower bound schemes involving entropic flows in Proposition \ref{Prop:ASEP A<0}.  We also give lower bounds in the case $A=0$ in ASEP in Corollary \ref{Prop:ASEP A=0}.
The associated costs of these strategies would be lower bounds for a possible rate function $I_\gamma(A)$.

\section{Results}
\label{section results}
Define, for $A\in \R$, the tail sets
$$E_A = \left\{ X_N/N \geq A\right\} \ \ \text{and} \ \ F_A = \left\{X_N/N\leq A\right\}.$$
Recall that the rate function $I_\gamma$ given in \eqref{Tasep_rate_function} and \eqref{I_gamma} for TASEP and ASEP.

We now state lower and upper bounds for the liminf and limsup of
$\frac{1}{N}\log \P_\rho\left(B\right)$
for $B$ equal to $E_A$ or $F_A$ as $A$ varies over different ranges in Sections \ref{results: lb section} and \ref{results: ub section}.  In Section \ref{results: upper tail}, we state an `upper tail' large deviation principle with rate function $I_\gamma$.

\subsection{Lower bound optimal strategies}
\label{results: lb section}
We first describe `lower bound' (LB) strategies and calculate their costs in ASEP corresponding to $0<\gamma\leq 1$. 

Denote by $\B_{r, \ell}$ the law of the joint process $(\eta_{Nt}, X_{Nt})$, starting from $(\eta_0, X_0)$, where the unscaled jump rates of the tagged particle are $r(1-\eta(X+1))$ and $\ell(1-\eta(X-1))$ of moving right and left respectively.  The law of the process \eqref{joint}, speeded up by $N$ and starting from $(X_0, \eta_0)$, corresponds to $\B_{p,q}$ when $r=p$ and $\ell=q=1-p$.
Let also $p_A= pc$ and $q_A = q/c$ where $c = (A+ \sqrt{A^2+4pq})/2p$.

\begin{thm}[ASEP: $A\geq\gamma$]
\label{thm:ASEP A large}
For ASEP with $0<\gamma\leq 1$, and $0<\epsilon$, consider the profile
\begin{align}
\label{u-A-large}
\rho_{0,\epsilon}(u) = \left\{\begin{array}{ll}
\rho& {\rm for \ }u<0\\
0&{\rm for \ }0\leq u< A-\gamma+\epsilon\\
\frac{1}{2\gamma}(u-A+\gamma-\epsilon)& {\rm for \ }A-\gamma+\epsilon \leq u< A-\gamma+\epsilon+2\rho\gamma\\
\rho& {\rm for \ }u\geq A-\gamma+\epsilon + 2\rho\gamma.
\end{array}\right.
\end{align}

Then, under the modified process measure, with $A_\epsilon = A+\epsilon/2$,
$$d\Q_N = \frac{d\nu^{N}_{\rho_{0,\epsilon}(\cdot)}(\cdot|\eta(0)=1)}{d\nu_\rho(\cdot|\eta(0)=1)} \frac{d\B_{p_{A_\epsilon}, q_{A_\epsilon}}}{d\B_{p,q}} d\P_\rho$$
the event $E_A$ is typical, $\lim_{N\uparrow\infty}\Q_N(E_A)=1$, and
the entropy cost
$$\lim_{\epsilon\downarrow 0}\liminf_{N\uparrow\infty}\frac{1}{N}E_{\Q_N}\left[ \log \frac{d\P_\rho}{d\Q_N}\right] \geq -I_\gamma(A),$$
from which we deduce that
$$\liminf_{N\rightarrow\infty} \frac{1}{N}\log \P_\rho(E_A) \geq -I_\gamma(A).$$
\end{thm}

\begin{thm}[ASEP: $\gamma(1-\rho)<A< \gamma$]
\label{thm:ASEP A not so large}
For ASEP with $0<\gamma\leq 1$, and $0<\epsilon<\gamma - A$, 
consider the profile
\begin{align}
\label{u-A-not so large}
\rho_{0, \epsilon}(u) = \left\{\begin{array}{ll}
\rho& {\rm for \ }u< 0\\
1-\tfrac{A_\epsilon}{\gamma} &{\rm for \ } 0\leq u< \gamma-A_\epsilon\\
1-\tfrac{A_\epsilon}{\gamma} + \tfrac{\rho - 1 + A_\epsilon/\gamma}{2(A_\epsilon-\gamma+\rho\gamma)}(u-\gamma+A_\epsilon)& {\rm for \ }\gamma-A_\epsilon\leq u< A_\epsilon-\gamma + 2\rho\gamma\\
\rho& {\rm for \ }u\geq A_\epsilon-\gamma + 2\rho\gamma
\end{array}
\right.
\end{align}
with $A_\epsilon = A +\epsilon$.

Then, under the modified process measure
$$d\Q_N = \frac{d\nu^{N}_{\rho_{0, \epsilon}(\cdot)}(\cdot|\eta(0)=1)}{d\nu_\rho(\cdot|\eta(0)=1)} d\P_\rho$$
the event $E_A$ is typical, $\lim_{N\uparrow\infty}\Q_N(E_A)\rightarrow 1$, and
the entropy cost
$$\lim_{\epsilon\downarrow 0}\liminf_{N\uparrow\infty}\frac{1}{N}E_{\Q_N}\left[ \log\frac{d\P_\rho}{d\Q_N}\right] \geq -I_\gamma(A),$$
from which we deduce that
$$\liminf_{N\rightarrow\infty} \frac{1}{N}\log \P_\rho(E_A) \geq -I_\gamma(A).$$
\end{thm}

The next result is stated only for TASEP $(\gamma=1)$, with respect to events $F_A$ for $0\leq A<1-\rho$.  Define the `holes' process $\chi_t$ where $\chi_t(x) = 1-\eta_t(x)$ for $x\in \Z$.  Here, by the `particle-hole' relation, we see $\chi_t$ is an exclusion process with rates $\widehat{p}(j) = 1-p(j)$ for $j=\pm 1$; when $\eta_t$ is TASEP, then $\chi_t$ is also a TASEP but jumping only to the left.

\begin{thm}[TASEP: $0\leq A< 1-\rho$]
\label{thm:TASEP nonentropy}
For TASEP ($\gamma=1$), and $0\leq \epsilon<1$, consider the nonentropic solution to \eqref{Burgers}:
\begin{align}
\label{nonentropy evo}
\rho_\epsilon(t,u) = \left\{\begin{array}{ll}
\rho & {\rm for \ }u<(A_\epsilon-\rho)t\\
1-A_\epsilon& {\rm for \ }(A_\epsilon-\rho)t\leq u<\rho  + (A_\epsilon-\rho)t\\
\rho& {\rm for \ }u\geq \rho  +(A_\epsilon-\rho)t
\end{array}\right.
\end{align}
for $0\leq t\leq 1$ and $A_\epsilon = A(1-\epsilon)$.

For $0<A<1-\rho$, 
there are TASEP process measures $\mathcal{R}_N$, 

\noindent starting from $\nu^{N}_{\rho_\epsilon(0,\cdot)}(\cdot|\eta(0)=1)$, such that, for each $\delta>0$,
$$\lim_{N\rightarrow\infty} \mathcal{R}_N\left(\{\pi^N_t: 0\leq t\leq 1\} \in B_\delta(\rho_\epsilon du)\right) =1$$ 
where $B_\delta$ is a $\delta$-ball around $\rho_\epsilon(\cdot, u) du$ in the Skorohod topology.
Define, the modified process measure
$$d\Q_N = \frac{d\nu^{N}_{\rho_\epsilon(0,\cdot)}(\cdot|\eta(0)=1)}{d\nu_\rho(\cdot|\eta(0)=1)}\frac{d\mathcal{R}_N}{d\P_{\nu^N_{\rho_\epsilon(0,\cdot)}(\cdot|\eta(0)=1)}} d\P_\rho.$$

When $A=0$, let $\widehat{\mathcal{B}}_\epsilon$ be the law of the process speeded up by $N$, starting from $\nu^{N}_{\rho_0(0,\cdot)}(\cdot|\eta(0)=1)$, where the nearest `hole' to the right of $x=\lfloor\rho N\rfloor$ has unscaled left jump rate $\widehat{p}(-1)=\rho - \epsilon$ for $0<\epsilon<\rho$ (instead of $1$) and the same vanishing right jump rate, $\widehat{p}(1)=0$, up to macroscopic time $t=1$.  Form, in this case, the modified process measure
$$d\Q_N = \frac{d\nu^{N}_{\rho_0(0,\cdot)}(\cdot|\eta(0)=1)}{d\nu_\rho(\cdot|\eta(0)=1)}\frac{d\widehat{\mathcal{B}}_\epsilon}{d\P_{\nu^N_{\rho_0(0,\cdot)}(\cdot|\eta(0)=1)}} d\P_\rho.$$

In both cases, when $0\leq A<1-\rho$, the event $F_A$ is typical, $\lim_{N\uparrow\infty} \Q_N(F_A) =1$, and
the entropy cost
$$\lim_{\epsilon\downarrow 0}\liminf_{N\rightarrow\infty}\frac{1}{N}E_{\Q_N}\left[ \log\frac{d\P_\rho}{d\Q_N}\right] \geq -I_\gamma(A),$$
from which we deduce that
$$\liminf_{N\rightarrow\infty} \frac{1}{N}\log \P_\rho(F_A) \geq -I_\gamma(A).$$
\end{thm}

In more general ASEP, a similar profile as in Theorem \ref{thm:TASEP nonentropy} may be constructed.  If the large deviation cost of such a profile were known, as it is for TASEP (cf. Section \ref{nonentropy cost section}), then we could also state a version of Theorem \ref{thm:TASEP nonentropy} for ASEP involving a nonentropic evolution.  However, here we give here a different lower bound, likely not sharp, using an entropic flow.

\begin{prop}[ASEP: $0<A<\gamma(1-\rho)$]
\label{Prop:ASEP A small}
Consider ASEP with $0<\gamma\leq 1$, and $0<\epsilon<1$.  Let
\begin{align}
\label{u-A-middle}
\rho_{0, \epsilon}(u) = \left\{\begin{array}{ll}
\rho& {\rm for \ }u<0\\
1-{A_\epsilon}/{\gamma} & {\rm for \ }0\leq u< \gamma\\
\rho& {\rm for \ }u\geq \gamma
\end{array}\right.
\end{align}
with $A_\epsilon = A(1-\epsilon)$.
Then, under the modified process measure
$$d\Q_N = \frac{d\nu^{N}_{\rho_{0, \epsilon}(\cdot)}(\cdot|\eta(0)=1)}{d\nu_\rho(\cdot|\eta(0)=1)} d\P_\rho$$
the event $F_A$ is typical, $\lim_{N\uparrow\infty} \Q_N(F_A)=1$, and
the entropy cost
\begin{align*}
&\lim_{\epsilon\downarrow 0}\liminf_{N\uparrow\infty}\frac{1}{N}E_{\Q_N}\left[\log \frac{d\P_\rho}{d\Q_N}\right]\\
&\ \ \ \ \  \geq
-\gamma\left[\big(1-\tfrac{A}{\gamma}\big)\log \frac{1-A/\gamma}{\rho} + \frac{A}{\gamma}\log \frac{A}{\gamma(1-\rho)}\right] :=-\mathcal{J}_1(A),
\end{align*}
from which we deduce that
$$\liminf_{N\rightarrow\infty} \frac{1}{N}\log \P_\rho(F_A) \geq - \mathcal{J}_1(A).$$ 
\end{prop}

In ASEP, when $A<0$, we also give a lower bound large deviation cost, likely not sharp.

\begin{prop}[ASEP: $A<0$]
\label{Prop:ASEP A<0}
Consider ASEP with $0<\gamma<1$.  For $\epsilon>0$,
let
 \begin{align}
 \label{u-A<0}
 \rho_{0,\epsilon}(u) = \left\{\begin{array}{ll}
\rho& {\rm for \ }u<A-\gamma-\epsilon\\
0&{\rm for \ }A-\gamma-\epsilon\leq u< 0\\
\rho& {\rm for \ }u \geq 0.
\end{array}\right.
\end{align}
Then, under the modified process measure, with $A_{\epsilon} = A -\epsilon/2$,
$$d\Q_N = \frac{d\nu^{N}_{\rho_{0,\epsilon}(\cdot)}(\cdot|\eta(0)=1)}{d\nu_\rho(\cdot|\eta(0)=1)}  \frac{d\B_{p_{A_{\epsilon}}, q_{A_{\epsilon}}}}{d\B_{p,q}} d\P_\rho$$
the event $F_A$ is typical, $\lim_{N\uparrow\infty} \Q_N(F_A)=1$, and
the entropy cost
\begin{align*}
&\lim_{\epsilon\downarrow 0}\liminf_{N\uparrow\infty}\frac{1}{N}E_{\Q_N}\left[ \log\frac{d\P_\rho}{d\Q_N}\right]\\
&\ \ \ \ \ \ \ \ \ \ \  \geq -(A-\gamma)\log(1-\rho) -A\log c + pc + q/c -1 := -\mathcal{J}_2(A),
\end{align*}
from which we deduce that
$$\liminf_{N\rightarrow\infty} \frac{1}{N}\log \P_\rho(F_A) \geq -\mathcal{J}_2(A).
 $$
\end{prop}

Finally, we comment on the case $A=0$ in ASEP when $0<\gamma< 1$, corresponding to when the tagged particle does not move macroscopically at time $t=1$.  
From Proposition \ref{Prop:ASEP A small}, when $A\downarrow0$, we recover a lower bound cost of $-\gamma\log(\rho)$.  However, from Proposition \ref{Prop:ASEP A<0}, as $A\uparrow 0$, we find a lower bound cost of
$-\left[\gamma\log(1-\rho) - 2\sqrt{pq}+1\right]$, noting that $c=\sqrt{q/p}$ when $A=0$.
Indeed, we observe that if $\rho=1$, there would be no movement, and so the cost of staying put vanishes, in accord with the strategy when $A\downarrow 0$ (Proposition \ref{Prop:ASEP A small}).  However, in the dilute setting $\rho\sim 0$, the birth-death clock change (Proposition \ref{Prop:ASEP A<0}) seems to be a better strategy.

\begin{cor}[$A=0$]
\label{Prop:ASEP A=0}
Consider ASEP with $0<\gamma<1$.  Then,
$$\lim_{A\downarrow 0}\liminf_{N\uparrow\infty} \frac{1}{N}\log \P_\rho(F_A)
\geq \max\left\{\gamma\log(\rho), \gamma\log(1-\rho) + 2\sqrt{pq} -1\right\}.$$
\end{cor}

\subsection{Upper bound minimizations}
\label{results: ub section}
We now discuss that the lower bound strategies in ASEP, when $A>\gamma(1-\rho)$, may be viewed in terms of minimizations of the rate $K$ over a class of initial distributions.  In this way, we may identify the rate function $I_\gamma$ in this regime.

For $A\in \R$, let
\begin{equation}
\label{I^Z eqn}
\I^Z(A)=
A\log(c) -pc-q/c +1 
\end{equation}
where $c=\left(A + \sqrt{A^2 + 4pq}\right)/2p$.  Notice that $\I^Z(A)=0$ exactly when $A=\gamma$.  The function $\I^Z$ is the rate function of a birth-death process $Z$ at time $t=1$ on $\Z$ with birth and death rates $p$ and $q$ respectively, found by changing the rates $p$ and $q$ to $p'=pc$ and $q'=q/c$.  When $\gamma = p = 1$, $I^Z$ reduces to the rate function of a standard Poisson process.

In the following upperbounds, $K_1(v_1)$ and $K_2(v_2)$ are minimum values in the optimizations \ref{problem1} and \ref{problem2} discussed in the sequel.

\begin{thm}[$A\geq\gamma$]
\label{thm:UB A large}
Consider ASEP with $0<\gamma\leq 1$.  We have
\begin{eqnarray*}
&\limsup_{N\uparrow\infty} \frac{1}{N}\log \P_\rho\left(E_A\right)\leq -\I^Z(A) - K_1(v_1) = -I_\gamma(A).
\end{eqnarray*}
\end{thm}

\begin{thm}[$\gamma(1-\rho)<A<\gamma$]
\label{thm:UB A not so large}
Consider ASEP with $0<\gamma\leq 1$.  We have
\begin{eqnarray*}
&\limsup_{N\uparrow\infty} \frac{1}{N} \log\P_\rho\left(E_A\right)\leq - K_2(v_2) = -I_\gamma(A).
\end{eqnarray*}
\end{thm}

To describe the optimization problems, let 
\begin{align}
\label{G-def-intro}
G(z) = \left\{\begin{array}{ll}
\frac{\gamma}{4}\left(1+\frac{z}{\gamma}\right)^2 & {\rm  for \ }|z|\leq \gamma\\
0&  {\rm for \ } z<-\gamma\\
z&  {\rm for \ }z>\gamma.
\end{array}
\right.
\end{align}
The function $G$ is the convex conjugate of $L$ defined by
$L(\rho) = -\gamma\rho(1-\rho)$ for $0\leq\rho\leq 1$, and $L(\rho)=\infty$ otherwise (cf. Appendix \ref{Hopf Lax section}).
The following optimization problems correspond to when $A\geq \gamma$ and $\gamma(1-\rho)<A<\gamma$ respectively:
\begin{optimization}[$A\geq \gamma$]
\label{problem1}
Minimize the cost
$$K_1(v):=\int_{A-\gamma}^{A+\gamma} v'(u)\log \frac{v'(u)}{\rho} + \left(1-v'(u)\right)\log \frac{1-v'(u)}{1-\rho} du$$
over differentiable functions $v:[A-\gamma, A+\gamma] \rightarrow [0,\infty)$ such that $0\leq v'(u)\leq 1$, $v(A-\gamma)=0$ and $v(u)\leq G(u-A)$ for $u\in [A-\gamma, A+\gamma]$.
\end{optimization}

\begin{optimization}[$\gamma(1-\rho)<A<\gamma$]
\label{problem2}
Minimize the cost
$$K_2(v):=\int_{0}^{A+\gamma} v'(u)\log \frac{v'(u)}{\rho} + \left(1-v'(u)\right)\log \frac{1-v'(u)}{1-\rho} du$$
over differentiable functions $v:[0, A+\gamma] \rightarrow [0,\infty)$ such that $0\leq v'(u)\leq 1$, $v(0)=0$ and $v(u)\leq G(u-A)$ for $u\in [0, A+\gamma]$.
\end{optimization}

We now solve these optimization problems in Proposition \ref{Prop:UB} whose proof, in the context of a more general optimization problem, is given in Section \ref{calc var section}.

\begin{prop}
\label{Prop:UB}
In Problem \ref{problem1}, the unique minimizer is given by
$$v_1(u) = \left\{\begin{array}{ll}
G(u-A) &  {\rm for \ } A-\gamma\leq u\leq \gamma(2\rho-1)+A\\
\rho (u- \gamma(2\rho-1) + A) + G\big(\gamma(2\rho-1)\big) &  {\rm for \ } \gamma(2\rho -1) + A\leq u\leq A+\gamma.
\end{array}\right.
$$
Whereas in Problem \ref{problem2}, the unique minimizer is given by
$$v_2(u) = \left\{\begin{array}{ll}
\left(1-\frac{A}{\gamma}\right)x& {\rm for \ } 0\leq u\leq \gamma-A\\
G(u-A) &  {\rm for \ } \gamma - A \leq u\leq \gamma(2\rho-1)+A\\
\rho (u- \gamma(2\rho - 1) + A) + G\big(\gamma(2\rho-1)\big)&  {\rm for \ } \gamma(2\rho-1)+A\leq u\leq A+\gamma.
\end{array}\right.
$$
\end{prop}

\subsection{`Upper tail' large deviation principle}
\label{results: upper tail}
We observe that the upper bound costs $I_\gamma(A)$, when $A>\gamma(1-\rho)$, in Theorems \ref{thm:UB A large}, \ref{thm:UB A not so large}
match 
the lower bound costs in Theorems \ref{thm:ASEP A large}, \ref{thm:ASEP A not so large} in ASEP when $0<\gamma\leq 1$.

\begin{cor}
\label{thm: upper tail}
Consider ASEP with $0<\gamma\leq 1$.
The scaled positions $\{X_N/N\}$ satisfy `upper tail' large deviation limits:  For $A\geq \gamma(1-\rho)$,
\begin{align*}
\lim_{N\rightarrow\infty} \frac{1}{N}\log \P_\rho(E_A) = - I_\gamma(A).
\end{align*}

\end{cor}

\section{Preliminaries}
\label{section preliminaries}

Before going to the proofs of the main results, we will need to detail helpful preliminary ingredients in the following $7$ subsections.  The lowerbound results make use of the first $5$ subsections, while the upperbound results quote from the first $2$ and last $2$ subsections.

\subsection{Comment on torus to line reductions}
\label{torus to line section}
 We remark that working on the line $\R$, with respect to the tagged particle motion starting from the origin, is tantamount to working on a large torus $S^1$.  Moreover, we will have occasion to use several results in Jensen's thesis \cite{Jensen_thesis} for TASEP stated on the torus $S^1$.  
 
The standard idea is that far away particles cannot come near the origin in a finite time with significant probability:  One may couple the process on $\R$ outside a large interval $[-RN, RN]$ with a system of independent Poisson rate $1$ motions that move only to the nearest right if starting to the left of the interval, and only to the nearest left if starting to the right of the interval.  One can show the event $E$ that particles outside the interval, with $R$ large, do not come into a subinterval of size $R_0N$, with $R_0<R$, up to time $NT$, has superexponential cost, that is $\lim_{R\uparrow\infty}\lim_{N\uparrow\infty}\frac{1}{N}\log \P(E) = -\infty$.  See \cite{Kosygina}; a similar argument was also formulated in 
\cite{Landim-Yau} for diffusive motions.  
 
Moreover, from this coupling, in considering all the events involving the tagged particle in this article, we may assume that the initial density of particles vanishes, or indeed is distributed as convenient, outside $[-RN, RN]$, for large $R>0$.

\subsection{Connections with the current}
\label{current section}

Let $J_t(x)$ be the (integrated) current across the site $x\in \Z$.  That is, 
$J_t(x)$ is the difference of the number of particles which have crossed from site $x$ to $x+1$ and those which have crossed from $x+1$ to $x$ up to time $t>0$.  In particular, as the interaction is nearest-neighbor, $J_t(z)-J_t(z-1) = \eta_0(z)-\eta_t(z)$.  Then, we have the relation, for $x<y$,
\begin{align}
\label{current-mass-relation}
J_t(y)-J_t(x) = \sum_{z=x+1}^y \left(\eta_0(z) - \eta_t(z)\right).
\end{align}

Our convention is that the tagged particle starts at the origin.  Then, for $a\geq 0$,
$\{X_t> a\} = \big\{ J_t(a)  \geq \sum_{z=0}^a \eta_0(z)\big\}$
signifies that the tagged particle being above $a$ at time $t$ is the same as all the particles initially between $0$ and $a$ (including the tagged particle) have crossed $a$ at time $t$.

When we start from a configuration that is empty to the left eventually, that is, $\eta_0(z)=0$ for all $z<0$ and $|z|$ large, 
since $J_t(a)= J_t(a)-J_t(0) +J_t(0)$ and $J_t(0) = \sum_{z\leq 0}\big(\eta_0(z)-\eta_t(z)\big)$ by \eqref{current-mass-relation}, we deduce for $a>0$, that
\begin{align}
\label{current-tagged}
\{X_t>a\} = \Big\{ \sum_{z\leq a}\eta_t(z) \leq \sum_{z\leq -1}\eta_0(z)\Big\}.
\end{align}
When $a<0$, the corresponding relations are
$\{X_t<a\} = \big\{ -J_t(a-1) \geq \sum_{z=a}^0 \eta_0(z)\big\}
= \big\{\sum_{z\leq a-1}\eta_t(z) \geq \sum_{z\leq 0}\eta_0(z)\big\}$.

\subsection{Basic martingales}
\label{martingale section}
We discuss two martingales associated to the tagged particle motion.  Let $N_+(t)$ and $N_-(t)$ be the number of jumps to the right and left made by the tagged particle up to time $t$.  Noting \eqref{joint}, these count processes are compensated by $\int_0^t p(1-\eta_s(X_s+1))ds$ and $\int_0^t q(1-\eta_s(X_s-1))ds$ respectively.  Then,
\begin{align*}
X_t &= N_+(t) - N_-(t)\\
&= \int_0^t \big\{p (1-\eta_s(X_s+1)) - q(1-\eta_s(X_s-1))\big\}ds + M_1(t)
\end{align*}
where $M_1$ is a martingale with bounded quadratic variation
$$\int_0^t \big\{p (1-\eta_s(X_s+1)) + q(1-\eta_s(X_s-1))\big\}ds \leq (p+q)t.$$
We may also use such a martingale, 
when the tagged particle has rates $p'$ and $q'$ different from the $p$ and $q$ rates of the other particles, via a modification of \eqref{joint}.

On the other hand, through the counting process representation, an exponential martingale may also be formed that will also be useful in the sequel.  Namely,
\begin{align*}
\exp\left\{\lambda X_t - \int_0^t \big\{p(e^\lambda -1) (1-\eta_s(X_s+1)) - q(e^{-\lambda}-1)(1-\eta_s(X_s-1))\big\}ds \right\} &= M_2(t)
\end{align*}
is a martingale for $\lambda\in \R$.

\subsection{LLNs for the tagged particle motion}
\label{LLN section}
To compute the lower bound large deviations of $X_t$, it will be helpful to understand the law of large numbers behavior under different environments.
Proofs of the following Propositions \ref{BD_LLN_prop} and \ref{BD_LLN_prop1} are deferred to Sections \ref{LLN appendix}.

We start with a limit when the tagged particle, initially at the origin, is carried by a density of particles.
\begin{prop}
\label{BD_LLN_prop}
Consider ASEP with $0<\gamma\leq 1$.  Let 
$\rho(\cdot,\cdot):[0,1]\times\R \rightarrow \R_+$ be a measurable function such that $\rho(s,u) =0$ for all $s\in [0,1]$ and $u<0$ with $|u|$ large.  
Suppose for $t\in [0,1]$ there is a unique solution $v_t=A\in \R$ such that
\begin{equation}
\label{LLN soln for v}
\int_{-\infty}^{A} \rho(t,u)du = \int_{-\infty}^0 \rho_0(u)du.
\end{equation}
Suppose also, with respect to a sequence of process measures $\P_N$ that
$\pi^N_t \rightarrow \rho(t, u)du$ in probability in $M(\R)$.

Then, in probability, as $N\rightarrow\infty$,
$\frac{1}{N}X_{Nt} \rightarrow v_t$.
\end{prop}

\begin{rem}
\label{entropic rem}
\rm We have the following comments.

1. When $\rho(\cdot, \cdot)$ is the entropic hydrodynamic density, and $\int_0^{\pm \epsilon}\rho_0(u)du \neq 0$, for an $\epsilon>0$, the statement of Proposition \ref{BD_LLN_prop} is part of the results in \cite{Reza_LLN}.  In this setting, $v_t$ satisfies \eqref{ODE-reza}.

2. In the cases that we use Proposition \ref{BD_LLN_prop}, 
it will be clear by inspection that \eqref{LLN soln for v} has a unique solution $v_t$ and that $\pi^N_t$ converges to $\rho(t,u)du$ in probability.

\end{rem}

We will also need to understand the behavior of a tagged particle in situations when either behind it or ahead of it there is empty space.  Let $\eta_{step}$ be the step configuration where $\eta_{step}(x) =\one(x\leq 0)$. 

\begin{prop}
\label{BD_LLN_prop1}
Consider ASEP with $0<\gamma\leq 1$.  

1. Suppose initially we start from the step configuration $\eta_{step}$.  Then,
 in probability, as $N\rightarrow\infty$,
$\sup_{t\in [0,1]} \left|\frac{1}{N}X_{Nt} - \gamma t\right| \rightarrow 0$.

2.  Suppose initially we start from $\nu_{\rho_0(\cdot)}(\cdot|\eta(0)=1)$ where 
$$\rho_0(u) = \left\{\begin{array}{ll}
0&{\rm for \ }u\leq 0\\
{u}/(2\gamma) & {\rm for \ }0\leq u\leq 2\gamma \rho\\
\rho& {\rm for \ }u\geq 2\gamma \rho.
\end{array}\right.
$$
Then,
in probability, as $N\rightarrow\infty$,
$\sup_{t\in [0,1]}\left|\frac{1}{N}X_{Nt} - \gamma t\right| \rightarrow 0$.
\end{prop}

\begin{rem}
\rm
We comment that
Part 1 of Proposition \ref{BD_LLN_prop1} is implied by the fixed time $t$ fluctuation result in \cite{TW}[Theorem 2] and the functional tightness Lemma \ref{tightness lemma}.  
\end{rem}

\subsection{Cost of a nonentropy profile}
\label{nonentropy cost section}

    We will also have occasion to use an explicit calculation of $I_{JV}$ for a class of profiles. Recall the definition of $f$, $g$ and $h$ in Section \ref{connections section}.
In \cite{Jensen_thesis}[p. 7] (see also \cite{Vilensky}), it is shown that the nonentropic solution, for $1\geq L\geq R\geq 0$,
$$\rho^{(L,R)}(t,u) = \left\{\begin{array}{ll}
L & {\rm for\ }u< (1-R-L)t\\
R&{\rm for \ } u\geq (1-R-L)t
\end{array}\right.
$$
for $0\leq t\leq 1$, which corresponds to a step initial profile maintaining its shock at the Rankine-Huguniot speed 
$\frac{R(1-R) - L(1-L)}{R-L} = 1-R-L$,
has dynamical cost
\begin{align}
\label{IJV cost}
I_{JV}(\rho^{(L,R)})&=g(R)-g(L) - \frac{f(R)-f(L)}{R-L} \big(h(R)-h(L)\big)\nonumber \\
&=L-R + LR\log \frac{R}{L} + (1-L)(1-R)\log \frac{1-L}{1-R}.
\end{align}

\subsection{Large deviation estimate when starting from $\eta_{step}$}

Let $Z_t$ be the position of the lead, tagged particle in ASEP starting from the step profile $\eta_{step}(x) = \one(x\leq 0)$.
The following large deviation statement for $Z_t$
will be of use in the proof of the upperbound Theorem \ref{thm:UB A large} when $A\geq \gamma$, and may be of independent interest.  We defer its proof to Section \ref{LDP section}.

Recall the formula for $\I^Z$ given in \eqref{I^Z eqn}.

\begin{prop}
\label{BD_LDP_prop}
We have, for $A\geq \gamma$, that
$$\lim_{N\uparrow\infty} \frac{1}{N} \log  \P_{\eta_{step}}\left( Z_N\geq AN\right) \leq  
-\I^Z(A).$$
\end{prop}
We comment that a nontrivial bound when $A<\gamma$ is not to be expected given $Z_N/N \rightarrow \gamma$ in probability as stated in Part 1 of Proposition \ref{BD_LLN_prop1}.  We remark that it may be possible, although not tried here, to extract a large deviation statement from the nonasymptotic formula for the probability mass function of $Z_t$ given in \cite{TW}.

\subsection{A superexponential bound for the flux}
\label{superexponential subsection}
 A version of the following estimate for TASEP was shown in Lemmas 2.7, 2.9 in Jensen \cite{Jensen_thesis}.  We state here a generalization for ASEP below, useful in the upperbound results, and remark that its argument, with simple modifications, would also yield a statement for finite-range asymmetric exclusion.  The proof is deferred to Section \ref{section superexponential}.

 Suppose $\rho_0(z)=0$ for $|z|>R$ with $R$ large.  Let $\rho(\cdot, \cdot)$ be the entropic solution of \eqref{Burgers}, and 
$v(t,x) = \int_{-\infty}^x \rho(t, z)dz$ for $t\geq 0$ and $x\in \R$.  
For $\delta>0$, let also $B_\delta(\rho_0)$ be the $\delta$-ball around the measure $\rho_0(u)du$ with respect to the distance \eqref{distance}.

 \begin{lem}
\label{superexponential cost lemma}
 Consider ASEP with $0<\gamma\leq 1$.  
 Let $t\in [0, 1]$. For each $\varepsilon>0$, and $L\in \R$, there is an $\delta>0$ such that
 \begin{align}
\label{super exp statement}
\lim_{N\rightarrow\infty}\frac{1}{N} \log \sup_{\eta_0}
\P_{\eta_0}\Big( v(t, L) >\frac{1}{N}\sum_{x\leq LN} \eta_{Nt}(x)  + \varepsilon\Big) = -\infty.
 \end{align}
 Here, the supremum is over $\eta_0$ such that $\pi^N_0\in B_\delta(\rho_0)$ and $\eta_0(x)=0$ for $|x|>RN$.
 \end{lem}

In words, the average of particles to the left of $LN$, at time $t=1$ say, cannot be more than obtained by the unique entropic solution, $\int_{-\infty}^L \rho(t, z)dz$, starting from the corresponding macroscopic initial condition, without incurring superexponential cost.
  
	\section{Lowerbounds}
  \label{LB section}
 We give in the following $5$ subsections the proofs of Theorems \ref{thm:ASEP A large}, \ref{thm:ASEP A not so large}, \ref{thm:TASEP nonentropy}, and Propositions \ref{Prop:ASEP A small}, \ref{Prop:ASEP A<0}.
  
 \subsection{Proof of Theorem \ref{thm:ASEP A large}: ASEP $A\geq \gamma$}
 \label{LB A>1 sect}

 We break the argument into steps.
 \vskip .1cm
 
 {\it Step 1.}
 For the tagged particle to deviate $X_{N} \sim AN$ for $A\geq \gamma$, 
 the particles in front of the tagged particle should be beyond $AN$ at time $N$.  This suggests initial particle distributions empty in an interval to the right of the origin.
  In particular, consider the
 profile $\rho_{0,\epsilon}$, for $\epsilon>0$ small, given in \eqref{u-A-large},
 and initial distribution $\nu^N_{\rho_{0,\epsilon}(\cdot)}(\cdot|\eta(0)=1)$.
  
  Let $\rho_\epsilon(t,x)$ be the entropic solution with respect to the Burgers evolution \eqref{Burgers}
 starting from this initial profile, for $0<t\leq 1$,
 \begin{equation}
 \label{u_eps_evol}
 \rho_{\epsilon}(t,u) = \left\{\begin{array}{ll}
 \rho& {\rm for \ }u< \gamma(1-\rho)t\\
 0& {\rm for \ }\gamma(1-\rho)t\leq u< A+\epsilon - \gamma(1- t)\\
\frac{u-A-\epsilon+\gamma(1-t)}{2\gamma(1-t)}
&{\rm for \ }A+\epsilon -\gamma(1- t)\leq u< A+\epsilon+\gamma(1-t)(2\rho-1)\\
 \rho& {\rm for \ }u\geq A+\epsilon+\gamma(1-t)(2\rho -1).
 \end{array}
 \right.
 \end{equation}
 Here, limits of the piecewise evolution are straightforwardly found.  For instance, the initial shock at the origin by the Rankine-Hugoniot relation moves with velocity $\gamma(1-\rho)$.
Also, if $r(t)$ is the velocity of the corner in $\rho_\epsilon$ initially at $u=A -\gamma + \epsilon + 2\rho\gamma$, then $\rho(t, r(t))\equiv \rho$, or $r'(t)\partial_u \rho(t, r(t)) = \gamma (\partial_u \rho)[1-2\rho]$ giving $r'(t) = \gamma(1-2\rho)$.

 The scaled relative entropy cost of $\nu^N_{\rho_{0,\epsilon}(\cdot)}(\cdot | \eta(0)=1)$ with respect to $\nu_{\rho}(\cdot |\eta(0)=1)$ 
 is calculated:
 \begin{align}
 \label{A>1profilecost}
 &\lim_{N\rightarrow\infty}\frac{1}{N} E_{\nu_{\rho_{0,\epsilon}(\cdot)}(\cdot|\eta(0)=1)}\left[\log \frac{d\nu^N_{\rho_{0,\epsilon}(\cdot)}(\cdot | \eta(0)=1)}{d\nu_\rho(\cdot|\eta(0)=1)}\right]
= -(A+\epsilon)\log(1-\rho) - \gamma\rho.
\end{align}
\vskip .1cm

{\it Step 2.}
Under ASEP, the unimpeded tagged position $X_{Nt}$ is a birth-death process with rates $Np$ and $Nq$.  The scaled relative entropy cost of the birth-death process to achieve a deviation of $\bar AN$ at time $t=1$ is
$\I^Z(\bar A) = \bar A \log c - pc -q/c + 1$
with $c = \big(\bar A + \sqrt{ \bar A^2 + 4pq}\big)/2p$ (cf. \eqref{I^Z eqn}).  Note for $\bar A\geq \gamma$ that $c\geq 1$.

\vskip .1cm
{\it Step 3.}
Recall that  $\B_{p,q}$ is the process law of the joint process $(\eta_{Nt}, X_{Nt})$ and $\B_{p', q'}$ is the law when the unscaled rates of the tagged particle are changed to $p'$, $q'$ for right and left jumps up to time $t=1$.  Note that $p'-q'=\bar A$.  We make the choice $\bar A = A_\epsilon= A+\epsilon/2$.  The Radon-Nikodym derivative $\frac{d\B_{p',q'}}{d\B_{p,q}}$ can be computed (cf. Appendix 1.7 \cite{kl}) with $F(t, (\eta_t, X_t)) = X_t\log c$ as
\begin{align}
\label{Step 4 LB}
&\exp\Big\{ F(Nt, (\eta_{Nt}, X_{Nt})) - F(0, (\eta_0, X_0)) - \int_0^{Nt} ds e^{-F(s, (\eta_s, X_s))}\big(\partial_s + \mathcal{L}\big)e^{F(s, (\eta_s, X_s))}\Big\}\nonumber\\
&\ \ \ = c^{X_{Nt}} \exp\left\{ - \int_0^{Nt} \left[p(1-\eta_s(X_s+1))(c-1) + q(1-\eta_s(X_s-1))(c^{-1}-1)\right]ds\right\}. 
\end{align}

Recall that $E_{\Q_N}$ is  expectation with respect to
$d\Q_N =  
\frac{d\nu^{N}_{\rho_{0,\epsilon}(\cdot)}(\cdot|\eta(0)=1)}{d\nu_\rho(\cdot|\eta(0)=1)} \frac{d\B_{p_A, q_A}}{d\B_{p,q}} d\P_\rho$, the changed process measure.
We write
\begin{align*}
&\P_\rho(E_{A}) = 
E_{\Q_N}\Big[ \frac{d\nu_\rho(\cdot|\eta(0)=1)}{d\nu^N_{\rho_{0,\epsilon}(\cdot)}(\cdot|\eta(0)=1)} \frac{d\B_{p, q}}{d\B_{p',q'}} \one(E_{A})\Big].
\end{align*}

\vskip .1cm
{\it Step 4.}
We now aim to calculate further $d\B_{p,q}/d\B_{p',q'}$.  Under $\Q_N$, 
let $y_s$ be the position of the next particle to the right of the tagged particle at time $s$, which initially is above $(A+\epsilon-\gamma)N$.  
By coupling, 
$y_s$ is greater than the position $y^*_s$ of a particle that starts at $\lfloor(A+\epsilon-\gamma)N\rfloor$ with no particles to the left of it, and others distributed according to $\nu_{\rho^*_\epsilon}$ where $\rho^*_\epsilon(u) = \rho_{0,\epsilon}(u)\one(u>A-\gamma+\epsilon)$.  By Part 2 of Proposition \ref{BD_LLN_prop1}, 
with respect to the evolution \eqref{u_eps_evol}, 
we have
$\sup_{0\leq s\leq t}\left |\frac{1}{N}y^*_{Ns} - \left(A+\epsilon-\gamma + \gamma s\right)\right|$
vanishes in probability as $N\uparrow\infty$.

By coupling again, $X_s$ is less than the position $x^*_s$ of a particle, with the same $p', q'$ rates, started at the origin from a step profile $\eta_{step}$ of particles each also with $p', q'$ rates, with no particles in front.
By Part 1 of Proposition \ref{BD_LLN_prop1},
$B_1:= \sup_{0\leq s\leq 1} \left | \frac{1}{N}x^*_{Ns} - (p'-q')s\right|$
vanishes in probability as $N\uparrow\infty$.

We conclude that $X_{Ns} \leq (p'-q')Ns + B_1$. With the choice $\epsilon>0$ small so that $p'-q' = \bar A = A+\frac{\epsilon}{2}>\gamma$, $X_{Ns}$ is always away by a macroscopic distance $O(\epsilon)$ away from 
$y_s$ for $0\leq s\leq 1$ with high probability.

\vskip .1cm
{\it Step 5.}  Let $B_2 = \big\{ \sup_{0\leq s\leq 1} \eta_{Ns}(X_{Ns}+1) = 0\big\}$.
 For $\delta<\epsilon/4$, consider the restriction
$$E':=
\Big\{ \sup_{0\leq s\leq 1}\Big|\frac{1}{N}X_{Ns} - (p'-q')s\Big|<\delta\Big\} \cap B_2
\subset E_{A}.$$
By Step 4, $\Q_N(B_2)\rightarrow 1$ as $N\rightarrow\infty$.

Note that $\frac{1}{N}X_{Ns} - \int_0^s \big\{p'[1-\eta_{Nr}(X_{Nr}+1)] -q'[1-\eta_{Nr}(X_{Nr}-1)]\big\}dr = M_1(Ns)/N$ is a martingale with quadratic variation of order $O(1/N)$ (cf. Section \ref{martingale section}).  Then, with the right coordinate empty, 
$$\big[\frac{1}{N}X_{Ns} - (p' -q')s \big]^-= \big[q\int_0^s \eta_{Nr}(X_{Nr}-1)dr + \frac{M_1(Ns)}{N}\big]^- \leq \big[\frac{M_1(Ns)}{N}\big]^-=O\big(\frac{1}{\sqrt{N}}\big),$$
where $[a]^- = \max\{-a, 0\}$ denotes the negative part.
 By Doob's maximal inequality,
we have $\Q_N\big(\sup_{0\leq s\leq 1} \big[\frac{1}{N}X_{Ns} - (p'-q')s\big]^-<\delta) \rightarrow 1$.  

The previous high probability bound $X_{Ns}\leq (p'-q')Ns + \delta$, as $B_1<\delta$ gives control of the positive part.
Hence, we conclude $\Q_N(E')\rightarrow 1$ as $N\uparrow\infty$.

Recall \eqref{Step 4 LB}.  As $\eta_{Ns}(X_{Ns}+ 1) = 0$ for $0\leq s\leq 1$ on $E'$ and $\int_0^{Nt} q\eta_s(X_s-1)(c^{-1}-1)ds \leq 0$ as $c^{-1}\leq 1$ irrespective of the occupation of the left coordinate, we have
\begin{align*}
&\frac{1}{N}\log\frac{d\B_{p' ,q'}}{d\B_{p,q}} \one(E')
= \one(E') \Big\{\log(c) \frac{1}{N}X_{N}\\
&\ \  \ \ - \frac{1}{N}\int_0^{N} \left[p(1-\eta_s(X_s+1))(c-1) + q(1-\eta_s(X_s-1))(c^{-1}-1)\right]ds\Big\} \\
&\leq \one(E')\Big\{(\log c)\big[(p'-q') + O(\epsilon)\big] -(p'+q' -1)\Big\}.
\end{align*}
Also, $\frac{1}{N}\log\frac{d\B_{p' ,q'}}{d\B_{p,q}} \one(E'^c) \leq \big[(\log c) \frac{W_N}{N} -q(c^{-1}-1)\big]\one(E'^c)$, where $X_N\leq W_N$ is the Poisson number of jump attempts to the right by the tagged particle with mean $p'N$.
Note, by Schwarz inequality, $E_{\Q_N} \big[\big(W_N/N\big)\one(E'^c)\big] \leq C_{p'}\sqrt{Q_N(E'^c)}$.  

Then,
$\limsup_{N\rightarrow\infty}\frac{1}{N}E_{\Q_N}\big[\log  \frac{d\B_{p' ,q'}}{d\B_{p,q}} \one(E')\big]
 \leq  (\log c)\big[(p'-q') + O(\epsilon)\big] -(p'+q' -1)$,
 and
  $\limsup_{N\rightarrow\infty}\frac{1}{N}E_{\Q_N}\big[\log\frac{d\B_{p' ,q'}}{d\B_{p,q}} \one(E'^c)\big] 
 \leq 0$. 
 Together,
 \begin{align}
 \label{BD entropy cost}
 \limsup_{N\rightarrow\infty}\frac{1}{N}E_{\Q_N}\Big[\log\frac{d\B_{p' ,q'}}{d\B_{p,q}} \Big] \leq
  (\log c)\big[(p'-q') + O(\epsilon)\big] -(p'+q' -1).
  \end{align}

\vskip .1cm
{\it Step 6.}
 The argument is now standard:  By Jensen's inequality,
  \begin{align*}
& \frac{1}{N}\log \P_\rho(E_{A}) \\
&\geq \frac{1}{N}\log \Q_N(E') + \frac{1}{N}\frac{1}{\Q_N(E')}E_{\Q_N}\Big[ \log \Big[ \frac{d\nu_\rho(\cdot|\eta(0)=1)}{d\nu^N_{\rho_{0,\epsilon}(\cdot)}(\cdot|\eta(0)=1)} \frac{d\B_{p, q}}{d\B_{p',q'}}\Big]\one(E')\Big].
\end{align*}

Now,
\begin{align*}
&E_{\Q_N}\Big[ \log \Big[\frac{d\nu_\rho(\cdot|\eta(0)=1)}{d\nu^N_{\rho_{0,\epsilon}(\cdot)}(\cdot|\eta(0)=1)}  \frac{d\B_{p, q}}{d\B_{p',q'}} \Big] \one(E')\Big]
\geq E_{\Q_N}\Big[ \log \Big[ \frac{d\nu_\rho(\cdot|\eta(0)=1)}{d\nu^N_{\rho_{0,\epsilon}(\cdot)}(\cdot|\eta(0)=1)}  \frac{d\B_{p, q}}{d\B_{p',q'}} \Big]  \Big] \\
&\ \ \ \ \ \ \ \ \ \ \ \ \ - \Q_N(E'^c)\frac{1}{\Q_N(E'^c)}E_{\Q_N}\Big[ \log \Big[ \frac{d\nu_\rho(\cdot|\eta(0)=1)}{d\nu^N_{\rho_{0,\epsilon}(\cdot)}(\cdot|\eta(0)=1)}  \frac{d\B_{p, q}}{d\B_{p',q'}} \Big]  \one(E'^c)\Big].
\end{align*}
By Jensen's inequality once more,
\begin{align*}
\frac{-1}{\Q_N(E'^c)}E_{\Q_N}\Big[ \log \Big[ \frac{d\nu_\rho(\cdot|\eta(0)=1)}{d\nu^N_{\rho_{0,\epsilon}(\cdot)}(\cdot|\eta(0)=1)}  \frac{d\B_{p, q}}{d\B_{p',q'}} \Big]  \one(E'^c)\Big]
\geq \log \Q_N(E'^c) - \log \P_\rho(E'^c).
\end{align*}

\vskip .1cm

{\it Step 7.}
As $\Q_N(E')\rightarrow 1$, $\P_\rho(E'^c)\rightarrow 1$, 
and $\Q_N(E'^c)\left(\log \Q_N(E'^c)
- \log \P_\rho(E'^c)\right)
\rightarrow 0$,
by the relative entropy calculations \eqref{A>1profilecost} and \eqref{BD entropy cost},
we recover
\begin{align*}
\liminf_{N\rightarrow\infty}
&\frac{1}{N}\log \P_\rho(E_{A})\\
&\geq \liminf_{N\rightarrow\infty}\left\{ \frac{1}{N}E_{Q_N}\Big[\log \frac{d\nu_\rho(\cdot|\eta(0)=1)}{d\nu^N_{\rho_{0,\epsilon}(\cdot)}(\cdot|\eta(0)=1)}\Big] + \frac{1}{N}E_{\Q_N}\Big[ \log \frac{d\B_{p, q}}{d\B_{p',q'}}\Big]\right\}\\
&\geq (A+\epsilon)\log(1-\rho) + \gamma \rho
 -(A+\frac{\epsilon}{2})\log c + pc + q/c -1 + O(\epsilon) 
\end{align*}
where $c = A + \epsilon/2 + \sqrt{(A+\epsilon/2)^2 + 4pq}\big)/2p$.  Taking $\epsilon\downarrow 0$, we achieve $-I_\gamma(A)$.
\qed

 \subsection{Proof of Theorem \ref{thm:ASEP A not so large}: ASEP $\gamma(1-\rho)<A<\gamma$}
 \label{proof A not so large section}
 
 For $X_{N}$ to move to near $AN$ where $\gamma(1-\rho)\leq A< \gamma$, it turns out, in contrast to the case $A\geq \gamma$, that we do not need to modify the rates of the tagged particle.  By only modifying the initial distribution, we can find a strategy which achieves the desired large deviation cost.
 
Recall $A_\epsilon = A+\epsilon$ here for $0<\epsilon<\gamma-A$.  Consider the initial profile $\rho_{0,\epsilon}$ given in \eqref{u-A-not so large},
and the corresponding (entropic) hydrodynamic evolution \eqref{Burgers}, for times $0< t\leq 1$:
 \begin{align}
\label{A not so large evo}
\rho_\epsilon(t,x) = \left\{\begin{array}{ll}
 \rho&{\rm for \ }u<m(t)\\
 \rho + \frac{1-{A_\epsilon}/{\gamma} - \rho}{s(t) - m(t)}(u-m(t)) & {\rm for \ } m(t)\leq u < s(t)\\
 1-\frac{A_\epsilon}{\gamma} &{\rm for \ }s(t)\leq u< \ell(t)\\
1-\frac{A_\epsilon}{\gamma} + \frac{ \rho - 1+A_\epsilon/\gamma}{r(t)-\ell(t)}(u-\ell(t))& {\rm for \ }\ell(t) \leq u< r(t)\\
 \rho&{\rm for \ } u\geq r(t),
 \end{array}
 \right.
 \end{align}
 where the rarefaction (forming as $1-A_\epsilon/\gamma<\rho$) is bounded by $m(t) = -\gamma\rho t$ and $s(t) = A_\epsilon t$.
 Also,
$\ell(t) = \gamma -A_\epsilon + (2A_\epsilon-\gamma)t $ and $r(t) = A_\epsilon-\gamma+2\rho\gamma + (1-2\rho)\gamma t$.

 The scaled relative entropy cost of $\nu^N_{\rho_{0,\epsilon}(\cdot)}(\cdot |\eta(0)=1)$ with respect to $\nu_{\rho}(\cdot|\eta(0)=1)$ is found in the limit as $N\uparrow\infty$ and $\epsilon\downarrow 0$ as
 \begin{align}
 \label{1-rho<A<1}
&\lim_{\epsilon\downarrow 0} \lim_{N\rightarrow\infty}\frac{1}{N}E_{\nu_{\rho_0(\cdot)}(\cdot|\eta(0)=1)} \left[ \log \frac{d\nu^N_{\rho_0(\cdot)}(\cdot |\eta(0)=1)}{d\nu_{\rho}(\cdot|\eta(0)=1)}\right]\nonumber\\
&\ \ \ \ =A\log\frac{A}{\gamma(1-\rho)} -A + \gamma(1-\rho) \ = \ I_\gamma(A).
 \end{align}

Consider large torus approximations as discussed in Section \ref{torus to line section}.  In the Skorohod space $D([0,1], M(\R))$, with the initial condition $\nu^N_{\rho_{0, \epsilon}(\cdot)}(\cdot|\eta(0)=1)$, 
we have by the hydrodynamic limit \eqref{Burgers} and continuity at the endpoint $t=1$ that $\pi_1^N$ converges to $\rho_\epsilon(1,\cdot)du$ in probability. 
Then, by Proposition \ref{BD_LLN_prop},
 noting the hydrodynamic evolution \eqref{A not so large evo},
 the scaled tagged motion $\frac{1}{N}X_{N}$ converges to $v_1=A_\epsilon>A$ in probability.

 Then, we have 
 $\lim_{N\uparrow\infty}\P_{\nu^N_{\rho_{0, \epsilon}(\cdot)}(\cdot| \eta(0)=1)}\big(E_A\big)=1$. 
 Moreover, write
 \begin{align*}
 \P_\rho(E_A) &= \E_{\nu^N_{\rho_0(\cdot)}(\cdot|\eta(0)=1)}\Big[
 \frac{d\nu_\rho(\cdot|\eta(0)=1)}{d\nu^{N}_{\rho_{0, \epsilon}(\cdot)}(\cdot |\eta(0)=1)} \one(E_{A}) \Big].
 \end{align*} 
 By the calculation \eqref{1-rho<A<1} and Jensen's inequality, analogous to the argument in Step 6 in Section \ref{LB A>1 sect}, 
 we have
 $\lim_{\epsilon\downarrow 0}\liminf_{N\rightarrow\infty}
 \frac{1}{N}\log  \P_\rho(E_A) \geq -I_\gamma(A)$, 
finishing the proof. \qed

 \subsection{ Proof of Theorem \ref{thm:TASEP nonentropy}: TASEP $0<A<1-\rho$}
 \label{nonentropy lb section}

In TASEP ($\gamma =1$), for $X_N$ to slow down and be near $AN$, as discussed in the introduction, nonentropic solutions to the hydrodynamic equation are of relevance.  We consider first when $0<A<1-\rho$.

Recall $A_\epsilon = A(1-\epsilon)$ for $0<\epsilon<1$, and consider the nonentropic solution $\rho_\epsilon(t,u)$ of the hydrodynamic equation \eqref{Burgers} given in \eqref{nonentropy evo}, and its initial profile
 $$\rho_{0, \epsilon}(u) = \left\{\begin{array}{ll}
 \rho&{\rm for \ }u<0\\
 1-A_\epsilon&{\rm for \ }0\leq u<\rho\\
 \rho&{\rm for \ }u\geq \rho.
 \end{array}\right.
 $$
 In \eqref{nonentropy evo}, the entropic shock location $r(t) = [1-(1-A_\epsilon+\rho)]t = (A_\epsilon-\rho)t$ and the nonentropic one is $s(t)=\rho + [1-(1-A_\epsilon + \rho)]t = \rho+ (A_\epsilon-\rho)t$.

 The scaled relative entropy of $\nu^N_{\rho_{0, \epsilon}(\cdot)}(\cdot|\eta(0)=1)$ with respect to $\nu_\rho(\cdot|\eta(0)=1)$ satisfies
 \begin{align}
\label{nonentropy rel entropy}
&\lim_{\epsilon\downarrow 0}\lim_{N\rightarrow\infty} \frac{1}{N} E_{\nu_{\rho_{0, \epsilon}(\cdot)}(\cdot|\eta(0)=1)} \Big[ \log \frac{d\nu^N_{\rho_{0, \epsilon}(\cdot)}(\cdot|\eta(0)=1)}{d\nu_\rho(\cdot|\eta(0)=1)}\Big]\nonumber \\
&\ \ \ \ \ \ \ \ \ \ \ \ \ \  = \rho\Big[(1-A)\log \frac{1-A}{\rho} + A\log \frac{A}{1-\rho}\Big].
\end{align}

Define now
 $\O = \big\{ \{\pi^N_t: 0\leq t\leq 1\} \in B_\delta(\rho_\epsilon(\cdot, u)du)\big\}$
 where $B_\delta$ is a ball of small radius $\delta>0$ around the measure-valued trajectory $\{\rho_\epsilon(t,u)du: 0\leq t\leq 1\}$ with respect to the Skorohod topology $D([0,1], M(\R))$.
  By Jensen's thesis \cite{Jensen_thesis}[Theorem 6.1] 
	(applied with the origin shifted to $\rho$, $s=A_\epsilon-\rho$ and $u = 1-A_\epsilon$),
there is a sequence $\{\mathcal{R}_N\}$ of TASEP process measures, say starting from initial distributions $\{\nu^N_{\rho_{0, \epsilon}(\cdot)}(\cdot|\eta(0)=1)\}$ such that 
 \begin{align}
 \label{R_N limit}
 \lim_{N\uparrow\infty} \mathcal{R}_N(\O)=1
 \end{align}
 and the scaled relative entropy satisfies
 \begin{align}
 \label{nonentropy cost}
& \lim_{\epsilon\downarrow 0}\lim_{N\rightarrow\infty} \frac{1}{N}E_{\mathcal{R}_N} \left[\log \frac{d\mathcal{R}_N}{d\P_{\nu^N_{\rho_{0,\epsilon}(\cdot)}(\cdot |\eta(0)=1)}}\right] \nonumber \\
&\ \ \ \ \ \ \ \ \ \ \  =1-A-\rho + (1-A)\rho\log\frac{\rho}{1-A} +A(1-\rho)\log \frac{A}{1-\rho}
\end{align}
by the evaluation in \eqref{IJV cost},
 with $L=1-A_\epsilon$ and $R=\rho$,
of the large deviation cost of this nonentropic profile; see Section \ref{nonentropy cost section}.

 By the discussion in Section \ref{torus to line section}, we may approximate the process by one on a large torus. Then, for $t=1$, the equation \eqref{LLN soln for v}, in terms of \eqref{nonentropy evo}, has a unique solution $v_1 = A_\epsilon$.  Also, by \eqref{R_N limit} and continuity of $\pi^N_t$ at the endpoint $t=1$, $\pi^N_1$ concentrates around $\rho_\epsilon(1,\cdot)du$ under $\mathcal{R}_N$.
 By 
  Proposition \ref{BD_LLN_prop} with $\delta, \epsilon>0$ small, we have $\O\subset F_A$ and 
 $\lim_{N\rightarrow\infty} \mathcal{R}_N\big(F_A\big) \geq \lim_{N\rightarrow\infty}\mathcal{R}_N\big(\O) = 1$.
 
Recall
$d\Q_N = \frac{d\nu^{N}_{\rho_{0, \epsilon}(\cdot)}(\cdot|\eta(0)=1)}{d\nu_\rho(\cdot|\eta(0)=1)}\frac{d\mathcal{R}_N}{d\P_{\nu^N_{\rho_{0, \epsilon}(\cdot)}(\cdot|\eta(0)=1)}} d\P_\rho$.  Observe that the sum of entropies \eqref{nonentropy rel entropy}, \eqref{nonentropy cost} equals $I_1(A)$, and $\lim_{N\rightarrow\infty}Q_N(\O)=\lim_{N\rightarrow \infty}\mathcal{R}_N(\O)=1$.
Given
 \begin{align*}
 P_\rho(F_{A}) 
 &\geq Q_N(\O) \frac{1}{Q_N(\O)} \int \one(\O) \frac{d\nu_\rho(\cdot|\eta(0)=1)}{d\nu^N_{\rho_{0, \epsilon}(\cdot)}(\cdot|\eta(0)=1)}\frac{dP_{\nu^N_{\rho_{0, \epsilon}(\cdot)}(\cdot|\eta(0)=1)} }{dQ_N} dQ_N,
  \end{align*}
we may
conclude
 $\lim_{\epsilon\downarrow 0}\liminf_{N\rightarrow\infty}\frac{1}{N}\log \P_\rho(F_A) \geq  -I_1(A)$, by
following the same path as in Step 6 of Section \ref{LB A>1 sect}.

We now discuss when $A=0$.  Under the change of measure $Q_N$, the tagged particle, initially at the origin and blocked by particles at all sites between $1$ and $\lfloor \rho N\rfloor$, cannot move until the `hole', called $y_\cdot$, initially nearest to the right of $\lfloor \rho N\rfloor$, reaches the origin.  This `hole', moving purely to the left, is not influenced before it reaches the origin.  The number of its jumps has the statistics of a Poisson process with rate $N(\rho-\epsilon)$ up to time $t=1$. So, as $y_0$ is near $\lfloor \rho N\rfloor$, $y_{N\cdot}$ is typically away from the origin:  $\sup_{0\leq s\leq t}|y_{Ns}/N - \rho(1-s) -\epsilon s| \rightarrow 0$ in probability with respect to $Q_N$.  Correspondingly, $Q_N(X_N/N = 0) \rightarrow 1$.    Now, following the same ideas as in Section \ref{LB A>1 sect}, the scaled large deviation cost of changing the initial distribution is $-\rho\log \rho$, and the scaled cost of changing the rates of the `hole' from a Poisson process with rate $1$ to rate $\rho - \epsilon$ is $(\rho - \epsilon) \log (\rho-\epsilon) -(\rho-\epsilon) + 1$.  Adding these, taking $\epsilon \downarrow 0$, gives $I_1(0) = 1-\rho$, from which we conclude $\liminf_{N\uparrow\infty} \frac{1}{N}\log \P_\rho(X_N/N = 0) \geq -I_1(0)$.
\qed

 \subsection{Proof of Proposition \ref{Prop:ASEP A small}: ASEP $0< A<\gamma(1-\rho)$}

 We follow a similar strategy as in the previous subsection.  However, for ASEP ($0<\gamma\leq 1$), we give a lower bound based on entropic notions, as large deviation lower bounds in general ASEP for the hydrodynamical evolution are not currently in hand.
 
 Recall $A_\epsilon = A(1-\epsilon)$ for $0<\epsilon<1$.  Consider the initial profile $\rho_{0, \epsilon}$ given in \eqref{u-A-middle}.
 Entropically, the density evolves for $0<t\leq 1$
as
 \begin{align*}
\rho_\epsilon(t,u)
&=\left\{ \begin{array}{ll}
 \rho& {\rm for \ }u<  (A_\epsilon-\rho\gamma)t\\
 1-A_\epsilon/\gamma&{\rm for \ } (A_\epsilon-\rho\gamma)t\leq u< \gamma -(\gamma - A_\epsilon)t\\
 \rho+\tfrac{\rho - (1-A_\epsilon/\gamma)}{t\big[\gamma-A_\epsilon + \gamma(1-\rho)\big]}\big(u-\gamma-\gamma(1-\rho)t\big)
 &{\rm for  }-(\gamma-A_\epsilon)t\leq u-\gamma< \gamma(1-\rho)t\\
 \rho& {\rm for \ }u\geq \gamma +\gamma(1-\rho)t.
 \end{array} \right. 
 \end{align*}
Recall the form of $\mathcal{J}_1(A)$ in the statement of 
Proposition \ref{Prop:ASEP A small}.  The scaled relative entropy cost of $\nu^N_{\rho_{0, \epsilon}(\cdot)}(\cdot|\eta(0)=1)$ with respect to $\nu_\rho(\cdot|\eta(0)=1)$ satisfies
\begin{align*}
&\lim_{\epsilon\downarrow 0}\liminf_{N\rightarrow\infty} \frac{1}{N}E_{\nu_{\rho_{0, \epsilon}(\cdot)}(\cdot|\eta(0)=1)}\left[ \log \frac{d\nu^N_{\rho_{0, \epsilon}(\cdot)}(\cdot|\eta(0)=1)}{d\nu_\rho(\cdot|\eta(0)=1)}\right]  = \mathcal{J}_1(A).
 \end{align*}

Consider large torus approximations as discussed in Section \ref{torus to line section}.  Under initial condition $\nu^N_{\rho_{0, \epsilon}(\cdot)}(\cdot|\eta(0)=1)$, by the hydrodynamic limit \eqref{Burgers} in $D([0,1], M(\R))$ and continuity at the endpoint $t=1$, $\pi^N_1$ converges to $\rho_\epsilon(1,\cdot)du$ in probability.   Noting the above evolution 
$\rho_\epsilon(t, u)$, 
there is a unique $v_1=A_\epsilon$ satisfying \eqref{LLN soln for v}.  Then, by Proposition \ref{BD_LLN_prop},
 $X_{N}/N$ converges to $A_\epsilon<A$ in probability.

 Hence, we have
 $\lim_{\epsilon\downarrow 0}\lim_{N\rightarrow\infty}\frac{1}{N}P_\rho\left(F_A\right)\geq - \mathcal{J}_1(A)$ by following straightforwardly the procedure in Section \ref{LB A>1 sect}.
To be brief, we do not repeat these details.  \qed

 \subsection{Proof of Proposition \ref{Prop:ASEP A<0}: ASEP $A<0$}
 Here, we discuss a lower bound for ASEP when $0<\gamma<1$, since in TASEP ($\gamma =1$) a deviation of $X_N \sim AN$ for $A<0$ is impossible.
  Recall, for $\epsilon>0$, the initial profile $\rho_{0,\epsilon}$ given in \eqref{u-A<0}.
 and consider its entropic evolution for $0<t\leq 1$, where the first shock devolves as a rarefaction wave and the second shock maintains its structure,
 \begin{align*}
 \rho_\epsilon(t,u) = \left\{\begin{array}{ll}
 \rho& {\rm for \ }u<A-\gamma - \epsilon-\rho\gamma t\\
 \frac{-\rho}{\gamma(1+\rho)t}(u-A+\gamma +\epsilon -\gamma t)& {\rm for \ }A-\gamma - \epsilon-\rho\gamma t \leq u< A-\gamma -\epsilon +\gamma t\\
 0&{\rm for \ }A-\gamma -\epsilon +\gamma t\leq u< \gamma(1-\rho)t\\
 \rho&{\rm for \ } u\geq \gamma(1-\rho)t.
 \end{array}
 \right.
 \end{align*}
 
 The scaled relative entropy cost of the initial profile, as $N\uparrow\infty$ is
 \begin{align}
\label{cost1}
\lim_{N\rightarrow\infty} \frac{1}{N}E_{\nu_{\rho_\epsilon(\cdot)}(\cdot|\eta(0)=1)}\left[\log \frac{d\nu^N_{\rho_\epsilon(\cdot)}(\cdot|\eta(0)=1)}{d\nu_\rho(\cdot|\eta(0)=1)}\right] = (A-\gamma -\epsilon)\log(1-\rho).
\end{align}
 
By the technique in Step 2 of Section \ref{LB A>1 sect}, we may change the rates of the tagged particle to $p'=pc$, $q'=q/c$ so that if it were the sole particle in the system, $\lim_{N\rightarrow\infty}\frac{1}{N}X_N = \bar A$ in probability at time $t=1$, with large deviation cost
 \begin{align}
\label{cost2}
 \bar A\log c - pc -q/c +1
\end{align}
 where $c = \big(\bar A + \sqrt{\bar A^2 + 4pq}\big)/2p$. Note for $\bar A\leq \gamma$ that $0<c\leq 1$. We will take $\bar A = A_\epsilon=A-\epsilon/2$.  
 
 We may follow the strategy in Section \ref{LB A>1 sect}. Let $y_\cdot$ be the location of the nearest particle to the left of the tagged particle.  This location cannot be more than the position of the lead particle starting from the step profile $\eta_0(x) = \one\big(x\leq (A-\gamma - \epsilon)N\big)$.  By Part 1 of Proposition \ref{BD_LLN_prop1}, the lead particle concentrates around $(A-\gamma-\epsilon)N + \gamma Ns$ at time $Ns$ with high probability.  Also, by coupling, $X_\cdot$ is greater than the leftmost particle $x^*_\cdot$ starting from a step profile $\eta^*_0 = \one(x\geq 0)$ where all particles have right and left rates $p'$ and $q'$.  Via Part 1 of Proposition \ref{BD_LLN_prop1}, noting that the drift is leftward, $x^*_{Ns}/N$ concentrates around $(p'-q')s$ for all $s\in [0,1]$ with high probability.  Therefore, $y_{Ns}$ is away a macroscopic distance from the tagged position $X_{Ns}$ for all $s\in [0,1]$.
 
 Hence, we may conclude, as in Section \ref{LB A>1 sect}, that $X_N/N \in F_A$ with high probability.  Recall $\mathcal{J}_2(A)$ in the statement of Proposition \ref{Prop:ASEP A<0}.  Adding the two costs \eqref{cost1}, \eqref{cost2}, we may conclude both $\lim_{\epsilon\downarrow 0}\liminf_{N\uparrow\infty}E_{\mathcal{Q}_N}\big[\log \frac{d\P_\rho}{d\mathcal{Q}_N}\big]\geq -\mathcal{J}_2(A)$ and
 $$
\lim_{\epsilon\rightarrow 0} \liminf_{N\uparrow\infty} \P_\rho\left(\frac{X_N}{N} \in F_A\right) \geq - \mathcal{J}_2(A).$$
 Again, to be brief, we do not repeat these details.  \qed

 \section{Upperbounds}
 \label{UB section}

 We concentrate on when the macroscopic displacement $A$ of the tagged particle at macroscopic time $t=1$ satisfies $A> \gamma(1-\rho)$.  After some initial discussion, we prove Theorems \ref{thm:UB A large} and \ref{thm:UB A not so large} in Sections \ref{UB subsection 1} and \ref{UB subsection 2} for the cases $A\geq \gamma$ and $\gamma(1-\rho)<A\leq \gamma$.
  \smallskip

{\it Torus approximation.} By remarks in Section \ref{torus to line section}, for the specified $A$, we may assume initially that there are no particles at $z$ when $|z|> \z N$ for a large $\z>0$, and the distribution on sites $-\z N<z<\z N$ is given by $\prod_{-\z N<z<\z N}{\rm Bern}(\rho)$.  We denote this initial distribution conditioned to have a particle at the origin by $\nu_{0,\rho}(\cdot|\eta(0)=1)$ and $\P_{0,\rho}$ as the associated process measure.
 \smallskip

{\it Current relation.} Recall the set $E_{A}= \big\{\frac{1}{N}X_N\geq A\big\}$, and the current relation \eqref{current-tagged}.
 We have
\begin{align}
\label{upp 1}
E_{A} \subset\Big\{ \sum_{y<AN} \eta_N(y) \leq  \sum_{y< 0} \eta_0(y)\Big\}.
\end{align}

\smallskip
{\it Large deviation bound with respect to $\eta_{step}$.}
Consider the tagged particle position $Z_\cdot$ starting from the step initial profile $\eta_{step}(x)=\one(x\geq 0)$. 
Note that $Z_\cdot$ does not depend on the initial configuration of particles $\eta_0$ and, 
by coupling, $Z \geq X$.  In particular, 
\begin{align}
\label{XlessZ}
\P_{\eta_0}(X_N>AN) \leq \P_{\eta_{step}}(Z_N>AN),
\end{align}
for which there is a nontrivial large deviation estimate in Proposition \ref{BD_LDP_prop} when $A\geq \gamma$.

\smallskip
{\it Cost of changing initial condition.}
Now, with respect to a continuous $\beta: \R \rightarrow (0,1)$, let
$\beta_0$ be the function
$$\beta_0(u) = \left\{\begin{array}{ll}
0& {\rm for \ }|u|> \z\\
\beta(u)&{\rm for \ } -\z\leq u\leq \z.
\end{array}\right.
$$
Form the local equilibrium measure $\nu^N_{\beta_0(\cdot)}(\cdot|\eta(0)=1)$, 
and let $\P_\beta$ and $\E_\beta$ be shorthand for the associated process measure and expectation, starting from $\nu^N_{\beta_0(\cdot)}(\cdot|\eta(0)=1)$.

Note that the Radon-Nikodym derivative equals
\begin{align*}
\frac{d\nu^N_{\beta_0(\cdot)}(\cdot|\eta(0)=1)}{d\nu_{0, \rho}(\cdot|\eta(0)=1)} = \prod_{-\z N\leq y \leq \z N} \left(\frac{\beta(y/N)}{\rho}\right)^{\eta_0(y)} \left(\frac{1-\beta(y/N)}{1-\rho}\right)^{1-\eta_0(y)}.
\end{align*}
Consider the functional $h^N_\beta(\pi_0^N)=\log \frac{d\nu^N_{\beta_0}}{d\nu_{0,\rho}}$ evaluated as
\begin{align}
\label{upp 2}
h^N_\beta(\pi^N_0)= \frac{1}{N}\sum_{-\z N\leq y\leq \z N} \eta_0(y)\log \left(\frac{\beta(y/N)}{\rho}\right) + (1-\eta_0(y))\log 
\left(\frac{1-\beta(y/N)}{1-\rho}\right).
\end{align}
We observe, uniformly over $\eta_0$, that $N|h^N_\beta(\pi^N_0)| = O(N)$.

Let $\mathcal{M}_{\z}$ be the compact set of measures $\mu$ on $[-\z,\z]$,
such that $\mu([a,b])\leq |b-a|$ for $-\z\leq a<b\leq \z$,
endowed with the distance \eqref{distance}.   Let also $\mathcal{M}_{ac}$ be the closed subset of absolutely continuous measures $W$ supported on $[-\z, \z]$ with density bounded by $1$.
  Note that $\frac{N}{N+1}\pi^N_t$ belongs to $\mathcal{M}_{\z}$, and so $\pi^N_t$ belongs to a $\delta$-envelope of $\mathcal{M}_{\z}$ for any $\delta>0$ for all large $N$.

For a measure $W\in \mathcal{M}_{\z}$, define 
\begin{align*}
h_\beta(W) &= \int_{-\z}^{\z} \log \frac{\beta(u)(1-\rho)}{(1-\beta(u))\rho}W(dx) + \int_{-\z}^{\z} \log \frac{1-\beta(u)}{1-\rho} du.
\end{align*}
It is known that the relative entropy $K_{-\z, \z}$ of a measure $W\in\mathcal{M}_{\z}$ with respect to $\rho\one(-\z\leq u\leq \z)du$ is infinite if $W$ is not absolutely continuous, and equals
$$K_{-\z, \z}(\theta) = \int_{-\z}^{\z} du \left\{ \theta(u)\log \frac{\theta(u)}{\rho} + (1-\theta(u))\log \frac{1-\theta(u)}{1-\rho}\right\}$$
when $W(du) = \theta(u)du$;
see \cite{kl}[Lemma V.5.2].  Also, by lower-semicontinuity, the level sets of $K_{-\z, \z}$ are closed and therefore compact subsets of $\mathcal{M}_{\z}$.

Observe, for all large $N$, when the distance of $\pi^N_0$ to $W(du)=\beta_0(u)du$ is within $\delta>0$, that
\begin{align}
\label{upp4}
h^N_\beta(\pi^N_0) \geq h_\beta(W) - 2\delta = K_{-\z, \z}(\beta_0) - 2\delta.
\end{align}

\smallskip
{\it Superexponential bound for the flux.}
 Consider the entropic solution $\rho(\cdot,\cdot)$ of \eqref{Burgers} starting from $\beta_0$.
 Recall $v_\beta(t, x)=\int_{-\infty}^x\rho(t,u)du$, for $t\geq 0$ and $x\in \R$, and also $B_\delta(\beta_0)$, for $\delta>0$, is the $\delta$-ball around the measure $\beta_0(u)du$ with respect to the distance \eqref{distance} on $\mathcal{M}_{\z}$.  

Then, by Lemma \ref{superexponential cost lemma}, for each $L>0$ and $\epsilon>0$, we may choose $\delta>0$ small so that
\begin{align}
\label{upp 3}
\limsup_{N\rightarrow\infty} \frac{1}{N}\log \P_\beta \left(v_\beta(1,L)>\frac{1}{N} \sum_{y \leq LN}  \eta_N(y) +\epsilon, \pi_0^N\in B_\delta(\beta_0)\right) = -\infty.
\end{align}

\subsection{Proof of Theorem \ref{thm:UB A large}: ASEP $A\geq \gamma$}
\label{UB subsection 1}
We break the proof into steps.  Up to and including Step 5, the value of $A$ is in range $A>\gamma(1-\rho)$.  Only in Step 6, do we specialize to when $A\geq \gamma$. 
\vskip .1cm

{\it Step 1.}  
With the remarks \eqref{upp 1}, \eqref{XlessZ}, \eqref{upp 2}, \eqref{upp4}, \eqref{upp 3}, we have
\begin{align}
\label{A>1 help}
&\limsup_{N\rightarrow\infty} \frac{1}{N}\log \P_{0,\rho}\big(X_N\geq AN, \pi^N_0\in B_\delta(\beta_0)\big)\\
&\ \ = \limsup_{N\rightarrow\infty} \frac{1}{N}\log \P_{0,\rho}\big(X_N\geq AN,
 \big\{ \frac{1}{N}\sum_{y<AN} \eta_N(y) \leq \frac{1}{N}\sum_{y<0}\eta_0(y)\big\}, \pi^N_0\in B_\delta(\beta_0)\big)\nonumber\\
&\ \ = \limsup_{N\rightarrow\infty} \frac{1}{N}\log E_{\beta^N_0}\Big[\exp\big\{-Nh^N_\beta(\pi^N_0)\big\}\one(\pi^N_0\in B_\delta(\beta_0)) \nonumber \\
&\ \ \ \ \  \times \one\big(X_N\geq AN\big)\one\big( v_\beta(1, A)-\epsilon< \frac{1}{N}\sum_{y< AN} \eta_N(y)\big)
 \one\big(\frac{1}{N}\sum_{y<AN} \eta_N(y) \leq \frac{1}{N}\sum_{y<0}\eta_0(y)\big) \Big]\nonumber\\
&\ \ \leq \limsup_{N\rightarrow\infty} \frac{1}{N}\log E_{\beta^N_0}\Big[\exp\big\{-Nh^N_\beta(\pi^N_0)\big\}\one(\pi^N_0\in B_\delta(\beta_0)) \nonumber \\
&\ \ \ \ \ \times \one\big( v_\beta(1, A)-\epsilon< \frac{1}{N}\sum_{y<0}\eta_0(y)\big) 
\P_{\eta_0}(X_N\geq AN)\Big]\nonumber\\
&\ \ \leq \limsup_{N\rightarrow\infty}\frac{1}{N}\log P_{\eta_{step}}(Z_N\geq AN) \nonumber\\
&\ \ \ \ \ \ \ - \one\left(v_\beta(1, A)-\epsilon< v_\beta(0,0) +\delta\right)\big[K_{-\z, \z}(\beta_0)-2\delta\big].\nonumber
\end{align}
Here, we bounded $\frac{1}{N}\log \big(\one\left(v_\beta(1, A)-\epsilon< v_\beta(0,0) +\delta\right)E_{\beta^N_0}\big[ \exp \big\{ -N(K_{-\z, \z}(\beta_0)-2\delta)
\big\}\big] \big)$ by $- \one\left(v_\beta(1, A)-\epsilon< v_\beta(0,0) +\delta\right)\big[K_{-\z, \z}(\beta_0)-2\delta\big]$.

\vskip .1cm
{\it Step 2.}
Since $\mathcal{M}_{\z}$ is compact,
from the open cover $\{B_\delta(\beta_0): \beta(\cdot):\R \rightarrow (0,1)\}$ of $\mathcal{M}_{\z}$, we may select a finite subcover $\{B_\delta(\beta^1_0), \ldots, B_\delta(\beta^j_0)\}$.  By applying \eqref{A>1 help} and subadditivity, we have
\begin{align*}
&\limsup_{N\rightarrow\infty} \log \P_{0,\rho}\big(X_N\geq AN\big)\\
&\ \ \  \leq \limsup_{N\rightarrow\infty}\frac{1}{N}\log \P_{\eta_{step}}(Z_N\geq AN) \\
&\ \ \ \ \ \ \ \ \ \  -\min_{1\leq i\leq j} \one\left(v_{\beta^i}(1, A)-\epsilon< v_{\beta^i}(0,0) +\delta\right)\big[K_{-\z, \z}(\beta^i_0)-2\delta\big].
\end{align*}

Minimizing over all $\beta_0$, we have the further bound
\begin{align}
\label{final}
&\limsup_{N\rightarrow\infty} \frac{1}{N}\log \P_{0,\rho}\big(X_N\geq AN\big) \\
&\ \ \leq \limsup_{N\rightarrow\infty}\frac{1}{N}\log \P_{\eta_{step}}(Z_N\geq AN) -\inf_{\beta_0: v_{\beta}(1, A)-\epsilon< v_{\beta}(0,0) +\delta}\big[K_{-\z, \z}(\beta_0)-2\delta\big].\nonumber
\end{align}

\vskip .1cm
{\it Step 3.}
We argue that
\begin{align}
\label{step 3 inf help}
\liminf_{\delta, \epsilon\downarrow 0} \inf_{\beta_0: v_{\beta}(1, A)-\epsilon-\delta< v_{\beta}(0,0)}\big\{K_{-\z, \z}(\beta_0)-2\delta\big\}
\geq \inf_{\beta_0: v_{\beta}(1, A)\leq v_{\beta}(0,0)} K_{-\z, \z}(\beta_0).
\end{align}

Note for the function $\beta_0(u)=0$ for $-\z\leq u\leq \z$ that $v_\beta(1,A) = v_\beta(0,0) = 0$ and $K_{-\z, \z}(\beta_0) = -2\z\log(1-\rho)<\infty$.  Then, the infima $\inf_{\beta_0: v_{\beta}(1, A)-\epsilon-\delta< v_{\beta}(0,0)}K_{-\z, \z}(\beta_0)$ are uniformly bounded over $\delta, \epsilon> 0$.  

For each $\delta, \epsilon>0$, let $\beta^{\delta, \epsilon}$ be a function, bounded by $1$, such that $v_{\beta^{\delta, \epsilon}}(1, A)-\epsilon-\delta < v_{\beta^{\delta, \epsilon}}(0,0)$ and
$$K_{-\z, \z}(\beta^{\delta, \epsilon}) <   \inf_{\beta_0: v_{\beta}(1, A)-\epsilon-\delta< v_{\beta}(0,0) }K_{-\z, \z}(\beta_0) +\delta.$$  
By compactness of the level set of $K_{-\z, \z}$ with respect to level $-2\z\log(1-\rho)$, extract a subsequence $\{\beta^k(u)du = \beta^{\epsilon_k,\delta_k}(u)du\}$ from the family $\{\beta^{\delta, \epsilon}(u)du\} \subset \mathcal{M}_{ac}$ converging to a $\beta^0(u)du\in \mathcal{M}_{ac}$ as $\epsilon^k, \delta^k\downarrow 0$.

Then, by lower-semicontinuity of $K_{-\z, \z}$, we have
\begin{align}
\label{upp 5} 
\liminf_{\delta, \epsilon\downarrow 0} \inf_{\beta_0: v_{\beta}(1, A)-\epsilon-\delta< v_{\beta}(0,0) }\big\{K_{-\z, \z}(\beta_0)-2\delta\big\} \geq \liminf_{k\rightarrow\infty} K_{-\z, \z}(\beta^k) \geq K_{-\z, \z}(\beta^0).
\end{align}
Also, as $\beta^k(u)du$ converges to $\beta^0(u)du$ in the distance \eqref{distance} as $k\rightarrow\infty$, we have for $y\in \R$ that
$$v_{\beta^k}(0,y) - v_{\beta^k}(0,0)= \int_0^y \beta^k(u)du \rightarrow v_{\beta^0}(0,y) - v_{\beta^0}(0,0).$$

Recall $G$ in \eqref{G-def-intro}.  Then, from the Hopf-Lax formulation (cf. Appendix \ref{Hopf Lax section}), and lower-semicontinuity with respect to the supremum,
\begin{align*}
0 &\geq \liminf_{k\rightarrow\infty} 
v_{\beta^k}(1, A) - v_{\beta^k}(0,0) -\epsilon^k-\delta^k\\
&= \liminf_{k\rightarrow\infty} \sup_y\left\{ v_{\beta^k}(0, y) - v_{\beta^k}(0,0) - G(y-A)\right\}  \\
&\geq \sup_y\left\{v_{\beta^0}(0, y) - v_{\beta^0}(0,0) - G(y-A)\right\} \
=\ v_{\beta^0}(1, A) - v_{\beta^0}(0,0).
\end{align*}

Then, the bound $0\geq v_{\beta^0}(1, A) - v_{\beta^0}(0,0)$ implies that we may lower bound in \eqref{upp 5},
$$K_{-\z, \z}(\beta^0) \geq \inf_{\beta_0: v_{\beta}(1, A)\leq v_{\beta}(0,0)}
K_{-\z, \z}(\beta_0).$$
Hence, \eqref{step 3 inf help} holds.

 \vskip .1cm
 {\it Step 4.}
In words, the restriction in the second infimum in \eqref{step 3 inf help} is that the initial condition $\pi(0, du) = \beta_0(u)du$ is such that the {\it entropic} solution $\pi(t,du)=\rho(t,u)du$ satisfies
\begin{align*}
0&\geq v_\beta(1,A) - v_\beta(0,0)
= \sup_y \big\{v_\beta(0,y) - v_\beta(0,0) - G\big(\frac{y-A}{t}\big)\big\}.
\end{align*}
 That is, a restriction is imposed on the density $\rho(0,\cdot)=\beta_0(\cdot)$:
 \begin{align}
\label{upp 6}
 v_\beta(0,y)-v_\beta(0,0)&=\int_0^y \rho(0,u)du\\
 & \leq G(y-A) \ 
= \left\{\begin{array}{ll} 0&{\rm for \ } y\leq A-\gamma\\
\frac{\gamma}{4}\big(1+\frac{y-A}{\gamma}\big)^2 &{\rm for \ }A-\gamma\leq y\leq A+\gamma\\
y-A&{\rm for \ } y\geq A+\gamma.
\end{array}\right.\nonumber
\end{align}

\vskip .1cm
{\it Step 5.}
We argue that there is no restriction on $\rho(0,u)$ for $u\leq 0$: 
For $y\leq 0$, as $\rho(0,\cdot)\geq 0$, the integral $\int_0^y\rho(0,u)du$ is non-positive, meeting the requirement of \eqref{upp 6}, no matter the value of $\rho(0,u)$ for $u\leq 0$.  

We also argue that there is no restriction on $\rho(0, u)$ for $u\geq A+\gamma$: No matter the value of $0\leq \rho(0, u)\leq 1$ for $u\geq A+\gamma$, we have $\int_{A+\gamma}^y \rho(0, u)du \leq y-A-\gamma$ for $y\geq A+\gamma$.   Note also $\int_{0}^{A-\gamma}\rho(0, u)du\leq 0$ andso  $\int_{A-\gamma}^{A+\gamma}\rho(0, u)du \leq \gamma$ by \eqref{upp 6}. Hence, the requirement in \eqref{upp 6}, for $y\geq A+\gamma$, is satisfied as
\begin{align*}
\int_0^y \rho(0, u)du &= \int_{0}^{A-\gamma}\rho(0, u)du + \int_{A-\gamma}^{A+\gamma}\rho(0, u)du + \int_{A+\gamma}^y \rho(0, u)du\\
&\leq 0 + \gamma + (y-A-\gamma) = y-A.
\end{align*}

Then, to minimize the entropy $K_{-\z, \z}$, we should take $\rho(0,u) = \rho$ when either $u\leq 0$ or $u\geq A+\gamma$.  Moreover, the restriction \eqref{upp 6} gives $\rho(0, u)=0$ when $0\leq u< A-\gamma$.

We arrive then at the following lowerbound for the right-hand side of \eqref{step 3 inf help}:
\begin{align}
\label{step 5 var expression}
&-\big[(A-\gamma)\vee 0\big]\log(1-\rho) \\
&\ \ \ \ \ \ +\inf_{v\in \K}  \int_{(A-\gamma) \vee 0}^{A+\gamma} v'(x)\log \frac{v'(x)}{\rho} + (1-v'(x))\log \frac{1-v'(x)}{1-\rho} dx \nonumber
\end{align}
where 
$$\K = \left\{0\leq v'(z)\leq 1 {\rm \  and  \ } v(0,z)-v(0,0)\leq G(z-A) \ {\rm for \ } (A-\gamma)\vee 0\leq z\leq A+\gamma\right\}.$$

\vskip .1cm

{\it Step 6.} The problem of identifying the unique argmin of the infimum in \eqref{step 5 var expression} was solved in Proposition \ref{Prop:UB} when $A\geq \gamma$.  Noting the discussion in Step 5, the value in \eqref{step 5 var expression} is therefore achieved when $v'(u)=\beta_0(u)$ is given by
$$\beta_0(u) = \left\{\begin{array}{ll}
\rho& {\rm for \ } u<0\\
0&  {\rm for \ }0\leq u< A-\gamma\\
\tfrac{1}{2\gamma}(\gamma-A+u)& {\rm for \ } A-\gamma\leq u < 2\rho\gamma +A-\gamma \\
\rho&  {\rm for \ } u\geq 2\rho\gamma +A-\gamma.
\end{array}
\right.
$$
The relative entropy cost of this initial profile was computed 
in \eqref{A>1profilecost}, with $\epsilon=0$, as
$K_{-\z, \z}(\beta_0)=-A\log(1-\rho) - \gamma\rho$.

Now, by Proposition \ref{BD_LDP_prop}, when $A\geq \gamma$, we have
$\limsup_{N\rightarrow\infty}\frac{1}{N}\log \P_{\eta_{step}}(Z_N/N \geq A) \leq -\I^Z(A)$.  Note that the sum of $\I^Z(A)$ and $K_{-\z, \z}(\beta_0)$ equals $I_\gamma(A)$.   Then, by considering \eqref{final} and \eqref{step 3 inf help}, we have the upper bound 
$\limsup_{N\rightarrow\infty} \frac{1}{N}\log \P_{0,\rho}\left(X_N \geq AN\right) \leq -I_\gamma(A)$,
the same as the lower bound in Section \ref{LB A>1 sect}. \qed

\subsection{Proof of Theorem \ref{thm:UB A not so large}: ASEP $\gamma(1-\rho)<A< \gamma$}
\label{UB subsection 2}

We follow the development up to the end of Step 5, and the scheme of Step 6, in the proof of Theorem \ref{thm:UB A large}.

Since $\gamma(1-\rho)<A< \gamma$, we have $(A-\gamma)\vee 0 = 0$.  Also, an argmin of the infimum in \eqref{step 5 var expression}, according to Proposition \ref{Prop:UB} when $\gamma(1-\rho)<A<\gamma$ and the discussion in Step 5, is given by $v'(u)=\beta_0(u)$ where
$$\beta_0(u) = \left\{\begin{array}{ll}
\rho&   {\rm for \ } u<0\\
1-\frac{A}{\gamma}&  {\rm for \ }0\leq u \leq \gamma -A\\
\tfrac{1}{2\gamma}(\gamma-A+u)& {\rm for \ } \gamma-A\leq u \leq 2\rho\gamma +A-\gamma \\
\rho&  {\rm for \ } u\geq 2\rho\gamma +A-\gamma.
\end{array}
\right.
$$

The relative entropy cost of this initial profile was computed in \eqref{1-rho<A<1} as $I_\gamma(A)$.  Note, trivially, 
$\limsup_{N\rightarrow\infty} \frac{1}{N}\log \P_{\eta_{step}}\big(Z_N\geq AN\big) \leq 0$.
Then, with respect to \eqref{final} and \eqref{step 3 inf help},  
$\limsup_{N\rightarrow\infty}\frac{1}{N}\log \P_{0,\rho}\left(X_N\geq AN\right) \leq -I_\gamma(A)$,
matching the lower bound derived in Section \ref{proof A not so large section}.  \qed
 
 \section{A calculus of variations problem and proof of Proposition \ref{Prop:UB}}
\label{calc var section}

We first discuss a general calculus of variations problem Proposition \ref{optimization prop A>1}, and deduce Proposition \ref{optimization prop A>1} in Section \ref{UB opt section}.  Then, the proof of Proposition \ref{optimization prop A>1} is given in Section \ref{general opt section}.

Let $a<b$ and $H:[a,b]\rightarrow [0,\infty)$ be a strictly convex, strictly increasing function.  Suppose that $H$ is continuously differentiable. If the relation $H'(u)=\rho$ has a solution in $[a,b]$, denote it by $z_\rho$.  Also, if the relation $H'(u) = H(u)/(u-a)$ has a solution in $[a,b]$, denote it by $y_{tan}$.

 \smallskip

 {\it Problem.}  Consider differentiable functions $v:[a,b]\rightarrow [0,\infty)$ such that $0\leq v'(u)\leq 1$ for $u\in [a, b]$.  Minimize the cost
 $$K_{a,b}(v)=\int_{a}^{b} v'(u)\log \frac{v'(u)}{\rho} + (1-v'(u))\log \frac{1-v'(u)}{1-\rho} du$$
 such that $v(a)=0$ and $v(u)\leq H(u)$ for $a\leq u\leq b$.

Note that since the cost is strictly convex and the constraint is convex, there is at most one minimizer $v$ to the Problem. 

In the following, we consider a few types of $H$, relevant in our context, specifically the second and fourth items below.  The other items flesh out the picture, but we leave a more general treatment to the interested reader.
\begin{prop}
\label{optimization prop A>1}
The Problem has a unique minimum $v$ with the following structure depending on the characteristics of $H$.

\begin{itemize}
\item When $H(u)\geq \rho u$ for $u\in [a,b]$, the minimum with zero cost is achieved when $v(u) = \rho u$.

\item When $H(a)=0$, $H'(a)<\rho$, and $z_\rho<b$, the minimum is achieved when
$$v(u) = \left\{\begin{array}{ll}
H(u)&{\rm  for \ } 0\leq u\leq z_\rho\\
\rho (u-z_\rho)+H(z_\rho)& {\rm  for \ } z_\rho\leq u\leq b.
\end{array}
\right.
$$

\item When $H(a)=0$, and $H'(b)\leq \rho$ (and therefore $H'(a)\leq \rho$), the minimum is achieved when $v(u) = H(u)$.

\item When $H(a)>0$, $H'(a)<\rho$, $y_{tan}<z_\rho$, and $z_\rho<b$, the minimum is achieved when
$$v(u) = \left\{\begin{array}{ll}
H'(y_{tan})(u-a)&  {\rm for \ } a\leq u\leq y_{tan}\\
H(u)&{\rm  for \ } y_{tan}\leq u\leq z_\rho\\
\rho(u-z_\rho) + H(z_\rho)& {\rm  for \ } z_\rho\leq u\leq b.
\end{array}
\right.
$$
\end{itemize}

\end{prop}

\subsection{Proof of Proposition \ref{Prop:UB}}
\label{UB opt section}
Recall $G$ in \eqref{G-def-intro}.   When $H(u) = G(u-A)$ on $[A-\gamma, A+\gamma]$, for $a=A-\gamma$, $b=A+\gamma$ and $A\geq \gamma$, we have $H(a)=0$, $H'(a) = 0<\rho$, and $a<z_\rho = \gamma(2\rho -1) + A<b$.  Hence, by the last part of Proposition \ref{optimization prop A>1}, the unique minimum of the Problem is achieved when 
 $$v(u) = \left\{\begin{array}{ll}
G(u-A)&{\rm  for \ } A-\gamma\leq u\leq \gamma(2\rho -1) +A\\
\rho (u-z_\rho) + G(z_\rho -A)& {\rm  for \ } \gamma(2\rho -1)+A\leq u\leq A+\gamma.
\end{array}
\right.
$$

Moreover, when $H(u) = G(u-A)$ on $[0, A+\gamma]$, with $a=0$, $b=A+\gamma$, and $\gamma(1-\rho)\leq A<\gamma$, we have $H'(a)=\frac{1}{2}\big(1-\frac{A}{\gamma}\big)<\rho$, $H(a)=\frac{\gamma}{4}\big(1-\frac{A}{\gamma}\big)^2>0$, $a<y_{tan}= \gamma -A< z_\rho=\gamma(2\rho-1)+A<b$, and $H'(y_{tan}) = 1-\frac{A}{\gamma}$.  Hence, by the second part of Proposition \ref{optimization prop A>1}, the unique minimum of the Problem is achieved when
$$v(u) = \left\{\begin{array}{ll}
\big(1-\frac{A}{\gamma}\big)u&  {\rm for \ } 0\leq u\leq y_{tan}\\
G(u-A)&{\rm  for \ } y_{tan}\leq u\leq z_\rho\\
\rho (u-z_\rho) + G(z_\rho -A)& {\rm  for \ } z_\rho\leq u\leq A+\gamma.
\end{array}
\right.
$$

Therefore, we recover the statement in Proposition \ref{Prop:UB}. \qed

  \subsection{Proof of Proposition \ref{optimization prop A>1}}
	\label{general opt section}
 To make use of an Euler-Lagrange formula, we consider a regularized problem which will help in the analysis of the Problem.  Note, for $\delta>0$, that both $H(a+\delta), H'(a+\delta)> 0$, as $H$ is strictly increasing.  This observation allows the introduction of $\epsilon>0$ in the following regularized problem.

 \smallskip

 {\it $\lambda$-Regularized Problem.}
 Fix $\lambda>0$ and $\delta>0$ small.  Let $0<\epsilon\leq H'(a+\delta)$.  Consider functions $v$ such that $v'(u)\in [\epsilon,1-\epsilon]$ for $u\in [a+\delta, b]$ and $0\leq v(a+\delta)\leq \min\{H(a+\delta), \delta\}$.
 Minimize
 $$\int_{a+\delta}^{b}  v'(u)\log \frac{v'(u)}{\rho} + (1-v'(u))\log \frac{1-v'(u)}{1-\rho} du 
 + \int_{a+\delta}^{b} \exp\big\{-\lambda\big( H(u)-v(u)\big)\big\}du.$$
 
\smallskip 
 \vskip .1cm
 We divide now the proof of Proposition \ref{optimization prop A>1} in steps.
 \vskip .1cm
 
 {\it Step 1.}  Consider the $\lambda$-Regularized Problem.
  Noting a priori that $\epsilon\leq v'(u)\leq 1-\epsilon$ for $u\in [a+\delta, b]$, and $0\leq v(a+\delta)\leq \delta$,
  we conclude $v(u)\geq 0$ is bounded by $b-a$ for $u\in [a,b]$.  So, for each $\lambda>0$, via equicontinuity of the admissible functions, there are minimizers $v_{\lambda, \delta, \epsilon}$ of the $\lambda$-Regularized Problem.  These are bounded, continuous, and such that $\epsilon \leq v'_{\lambda,\delta, \epsilon} \leq 1-\epsilon$ a.e. on $[a+\delta, b]$.

 \vskip .1cm
 {\it Step 2.}
 Let $h(z)=  z\log \frac{z}{\rho} + (1-z)\log \frac{1-z}{1-\rho}$ and note $h'(z)= \log \frac{z}{\rho} + \rho -\log \frac{1-z}{1-\rho} -(1-\rho)$ and $h''(z) = [z(1-z)]^{-1}$ for $z\in (0,1)$.  The Euler-Lagrange weak equation for the minimizers $v= v_{\lambda, \delta, \epsilon}$ is
 \begin{align*}
 \int_{a+\delta}^{b} h'(v'(u))\phi'(u)du = -\lambda \int_{a+\delta}^{b}
 \exp\big\{ -\lambda\big(H(u) - v(u)\big)\big\}\phi(u)du.
 \end{align*}
 Here, $\phi$ is a test function which vanishes at the boundaries $a+\delta$ and $b$.

Hence, the weak derivative of $h'(v'(\cdot))$ is computed,
 $$\frac{d^{(w)}}{du} h'(v'(u)) = \lambda \exp\big\{-\lambda\big(H(u)-v(u)\big)\big\},$$
 as a bounded, continuous function.
Therefore, the strong derivative of $h'(v'(u))$ exists and equals the weak derivative
 for all $u\in (a+\delta, b)$.    
 
 \vskip .1cm
 
 {\it Step 3.}
 Since $h''(z)$ is well-defined for $z\in [\epsilon,1-\epsilon]$, we conclude that the second derivative $v''(u)$ is well defined for $u\in[a+\delta, b]$.  Indeed,
 the quotient
 $$\frac{h'(v'(u+\kappa)) - h'(v'(u))}{\kappa} = h''(\bar v_{u,\kappa}) \frac{v'(u+\kappa) - v'(u)}{\kappa},$$
 with $\bar v_{u,\kappa}$ a number between $v'(u+\kappa)$ and $v'(u)$, would diverge as $\kappa\downarrow 0$ if $v'(u+\kappa) -v'(u)$ did not vanish, contradicting existence of the strong derivative of $h'(v'(u))$.  Then, if $\lim_{\kappa\downarrow 0} v'(u+\kappa) = v'(u)$, since the strong derivative of $h'(v'(u))$ is a priori bounded, we must have the limit $\lim_{\kappa\downarrow 0} \kappa^{-1}\left(v'(u+\kappa) - v'(u)\right) = v''(u)$ exists.
 
 Moreover, evaluating the derivative of $h'(v'(u))$, we have
 \begin{align}
 \label{EL eq}
 v''(u) &= \lambda [v'(u)(1-v'(u))] \exp\big\{ -\lambda\big(H(u) - v(u)\big)\big\} \geq 0
 \end{align}
 for $u\in (a+\delta, b)$.
 Therefore, $v$ must be convex on $[a+\delta,b]$.
 
 \vskip .1cm
 {\it Step 4.}  Consider the minimizers $\{v_{\lambda, \delta, \epsilon}\}_{\lambda\geq 1}$.  By convexity and equicontinuity of the family (as $\epsilon\leq v'_{\lambda, \delta, \epsilon}\leq 1-\epsilon$ a.e.), 
 let $v_*=v_{*,\delta,\epsilon}$ be a subsequential, uniformly converging limit point of $v_{\lambda,\delta, \epsilon}$ as $\lambda \uparrow\infty$.  The limit $v_*$ is also convex, continuous and bounded, and
we may arrange that $v'_*$ is the a.e. limit of $v'_{\lambda, \delta,\epsilon}$ on this subsequence. Note also that $0\leq v_*(a+\delta)\leq \min\{H(a+\delta),\delta\}$ (by the specification in the $\lambda$-Regularized Problem)
and $v'_*\in [\epsilon,1-\epsilon]$ a.e.

\vskip .1cm

{\it Step 5.} If $v_*(u) > H(u) + \kappa$ at some point $u\in (a+\delta, b)$ for $\kappa>0$, then $v_{\lambda, \delta,\epsilon}(u)> H(u) + \kappa/2$ for all large $\lambda$ in an interval about $u$. Then, by \eqref{EL eq}, $v''_{\lambda, \delta,\epsilon}(u)$ would diverge as $\lambda\uparrow\infty$ in an interval $(u, u+\upsilon)$ with $\upsilon>0$ small, a contradiction that $v'_*$ is bounded: Indeed, $1-\epsilon\geq v'_{\lambda, \delta, \epsilon}(u+\upsilon) - v'_{\lambda, \delta, \epsilon}(u) = \int_u^{u+\upsilon}v''_{\lambda, \delta, \epsilon}(s)ds$, but $v''_{\lambda, \delta, \epsilon}$ diverges as $\lambda\uparrow\infty$, contradicting that $1-\epsilon\geq v'_*(u+\upsilon) - v'_*(u)$.
 So, we may focus on when $v_*(u)\leq H(u)$ for all $u\in [a+\delta, b]$.

 On the other hand, if $v_*(u)\leq H(u)-\kappa$ for an $u\in [a+\delta, b]$ and so $v_*(u)\leq H(u)-\kappa/2$ in an interval containing $u$, then by \eqref{EL eq} we would conclude $\lim_{\lambda\uparrow\infty}v''_{\lambda, \delta, \epsilon}=0$ in this interval.  Therefore, in this case, $v_*$ is linear in this interval, and in particular $v''_*=0$ in this interval.
 
So, we are in the situation that $v_*$ is convex and $v_*(u) \leq H(u)$ for $u\in [a+\delta, b]$, and whenever $v_*(u)< H(u)$ in an interval that $v_*$ is linear in the interval.

\vskip .1cm
{\it Step 6.} When $v_*(a+\delta) < H(a+\delta)$, the optimizer $v_*$ must be linear on $[a+\delta, y_0]$ up to the point $y_0$ where it intersects the graph of $H(u)$, or if it doesn't intersect, then $y_0=b$.  If $y_0< b$, then by convexity of $v_*$, the line must be the tangent line to the graph of $H(u)$ going through the point $(a+\delta, v_*(a+\delta))$.  That is, $y_0$ satisfies
$
 H'(y_0)= \frac{H(y_0) - v_*(a+\delta)}{y_0-(a+\delta)}$.

\vskip .1cm
{\it Step 7.}
Again, by convexity, when $v_*(u_0)=H(u_0)$ at a point $u_0$, if the graph of $v_*$ on $[u_0, b]$ is not a tangent line,
the optimizer $v_*$ must equal $H$ for $u\in [u_0, z_0]$ up to some $u_0<z_0\leq b$, after which it dips below the graph of $H$.  If the point $z_0<b$, then $v_*$ must be linear on $[z_0,b]$ by the comments in Step 5. By convexity of $v_*$, the line must be a tangent line of $H$ with slope equal to $v'_*(z_0)=H'(z_0)$.

We have now described to an extent the structure of the optimizer 
$v_*=v_{*, \delta, \epsilon}$.  The initial value $v_*(a+\delta)$ and possible points $y_0$ and $z_0$ are still to be determined.

\vskip .1cm
{\it Step 8.}
For each $v_{*, \delta, \epsilon}$, continue it to $[a,b]$ by taking $v_{*, \delta, \epsilon}(u) = \min \{\frac{v_{*,\delta, \epsilon}(a+\delta)}{\delta}(u-a), H(u)\}$ for $u\in [a, a+\delta]$. 
Then, we have that $v_{*, \delta, \epsilon}(a)=0$ and, as $v_{*,\delta, \epsilon}(a+\delta)\leq \delta$, that $v'_{*, \delta, \epsilon}$ is a.e. bounded uniformly in $\delta$, $\epsilon$.

By equicontinuity of $\{v_{*, \delta, \epsilon}\}_{\delta, \epsilon>0}$, let $\bar v$ be a limit point of the family with respect to uniform convergence 
as $\epsilon\downarrow 0$ and then $\delta\downarrow 0$. The limit $\bar v$ is still convex, continuous and bounded, and we can arrange that $v'_{*, \delta, \epsilon}$ converges a.e to $\bar v'$.  

Given the form of the pre-limit functions $v_{*, \delta. \epsilon}$, the structure of $\bar v$ is that $\bar v(a)=0$ and $\bar v \leq H$ on $[a,b]$.  

Moreover, $\bar v$ may be initially linear up to a point $y_0$.  If $y_0<b$, then the line is a tangent line of $H$ with slope $H'(y_0)$; if $y_0=b$, then the line may be any line below the graph of $H$. 

If $y_0<b$, then $\bar v$ may follow the graph of $H$ on $[y_0, z_0]$, up to a point $z_0\geq y_0$.  

If $z_0<b$, then $\bar v$ will follow a tangent line of $H$ with slope $H'(z_0)$ on $[z_0,b]$; if $z_0=y_0$, $\bar v$ equals $H$ exactly at this one point.

The possible points $y_0$, $z_0$, or the slope of the line if $\bar v$ stays strictly below $H$ are to be determined.

\vskip .1cm

{\it Step 9.}
The strategy now is to 
see how the limit points $\bar v$ of the $\lambda$-Regularized Problem relate to the minimizer of Problem.
Indeed, we will show that $\bar v$ is a minimizer of the Problem and therefore, by uniqueness of solution to the Problem, is the unique limit point with respect to the $\lambda$-Regularized Problem.  

Let $w_*\leq H$ be admissible with respect to the Problem.  
Consider, for $\epsilon'>0$, 
$$w_{\delta, \epsilon'}(u)=\frac{1}{1+B\epsilon'}\big(w_*(u) +\epsilon' u).$$ 
We may choose $B=B_\delta$ so that
$w_{\delta, \epsilon'}(u)< H(u)$ for $u\in [a+\delta, b]$ and $w_{\delta, \epsilon'}(a+\delta)<\delta$.  Indeed,
\begin{align*}
H(u)-w_{\delta,\epsilon'}(u)&= \frac{1}{1+B\epsilon'}\big(H(u) - w_*(u)\big) + \frac{\epsilon'}{1+B\epsilon'}\big(BH(u) - u\big) \\
&\geq \frac{\epsilon'}{1+B\epsilon'}\big(BH(u) - u\big), \ \ {\rm and} \\
w_{\delta, \epsilon'}(a+ \delta) &\leq \frac{1}{1+B\epsilon'}\big(H(a+\delta) + \epsilon'(a+\delta)\big).
\end{align*}
We may choose $B$ large enough so that $BH(u)>u$ for $u\in [a+\delta, b]$ and $w_{\delta, \epsilon}(a+\delta) < \delta$.

By definition, $w'_{\delta, \epsilon'} \in [\tfrac{\epsilon'}{1+B\epsilon'}, \tfrac{1+\epsilon'}{1+B\epsilon'}]$.  When $\epsilon' = \tfrac{\epsilon}{1-B\epsilon}$, say for a given $\epsilon>0$, we have $w'_{\delta, \epsilon'}\geq \epsilon$. At the same time, as it can be arranged that $B\geq 2$,  
we have $w'_{\delta, \epsilon'}
 \leq 1-\epsilon$.  
Hence, $w_{\delta, \epsilon'}$, restricted to $u\in [a+\delta, b]$, is admissible in the constraints of the $\lambda$-Regularized Problem.  
\vskip .1cm

{\it Step 10.}
Then, as $v_{\lambda, \delta, \epsilon}$ is a minimizer of the $\lambda$-Regularized Problem,
\begin{align*}
&K_{a+\delta, b}(w_{\delta, \epsilon'}) + \int_{a+\delta}^{b} \exp\big\{ -\lambda\big(H(u) - w_{\delta, \epsilon'}(u)\big)\big\} du\\
&\geq K_{a+\delta, b}(v_{\lambda, \delta, \epsilon}) +\int_{a+\delta}^{b} \exp\big\{ -\lambda \big(H(u) - v_{\lambda, \delta, \epsilon}(u)\big)\big\}du
\ \geq \ K_{a+\delta, b}(v_{\lambda, \delta, \epsilon}).
\end{align*}
Since $w_{\delta, \epsilon'}(u)< H(u)$ and $v'_{\lambda, \delta, \epsilon} \rightarrow v'_{*,\delta, \epsilon}$ a.e. as $\lambda \rightarrow\infty$, by Fatou's lemma, we have as $\lambda\uparrow\infty$ that
$$K_{a+\delta, b}(w_{\delta, \epsilon'}) \geq K_{a+\delta, b}(v_{*,\delta, \epsilon}).$$

By the form of $w_{\delta, \epsilon'}$, the limit $\lim_{\delta, \epsilon\downarrow 0} K_{a+\delta, b}(w_{\delta, \epsilon'})=K_{a, b}(w_*)$.
Also, along a subsequence, as $\delta, \epsilon\downarrow 0$, with respect to a limit point $\bar v$ of $\{v_{*, \delta, \epsilon}\}$, we have by Fatou's lemma, that $\liminf_{\delta, \epsilon\downarrow 0}K_{a, b}(v_{*, \delta, \epsilon})$ is bounded below by $K_{a, b}(\bar v)$. 
Hence, $K_{a,b}(w_*)\geq K_{a, b}(\bar v)$.

Therefore, $\bar v$ is a minimizer and in fact the unique solution of the Problem, and  moreover is the unique limit point of the $\lambda$-Regularized Problem.

\vskip .1cm
{\it Step 11.}
 Since $\bar v$ is minimal with respect to the Problem, let us enumerate the possible structures for $\bar v$ 
as stated at the end of Step 8.  Given the forms of the cost $K_{a,b}$ and the convex function $H$, the following conclusions are now immediate.
 \smallskip
 
1. $H(u)\geq \rho u$ for $u\in [a,b]$. 
Here, $\bar v$ must be a line with smallest $K_{a,b}$ cost, namely the line with slope $\rho$ through $(a,0)$ with vanishing $K_{a,b}$ cost.
 
 2.  $H(a)=0$, $H'(a)<\rho$, and $z_\rho<b$.  Here, $\bar v$ follows the graph of $H$ up to point $z_\rho$ whereupon it follows the tangent line with slope $\rho$ to $b$.  The only other possibility is that $\bar v=0$ which has larger cost.
 
 3.  $H(a)=0$, $H'(b)\leq \rho$.  Here, $\bar v$ follows the graph of $H$ up to $b$.  Indeed, the slopes of $\bar v$ will be closer to $\rho$ than if $\bar v$ were linear with slope less than $H'(a)\leq \rho$, the only other possibility.

 4.  $H(a)>0$, $H'(a)<\rho$, and $y_{tan}<z_\rho<b$.  We claim that $\bar v$ is the tangent line up to $y_{tan}$ with slope $H'(y_{tan})$; then, it follows $H$ up to $z_\rho$, following after the tangent line with slope $\rho$ to $b$.  Indeed, $\bar v$ must be linear initially, and so is linear up to $y_{tan}$.  It is optimal now to follow the graph of $H$ up to $z_\rho$ as the slopes are closer to $\rho$ than if $\bar v$ were to continue on the tangent line with slope $H'(y_{tan})$.  Finally, departing on the tangent line with slope $H'(z_\rho)=\rho$ gives vanishing cost for the part of the integral $K_{a,b}$ from $z_\rho$ to $b$.
 \qed

\section{Proof of Proposition \ref{BD_LLN_prop}: LLNs for the tagged particle motion}
\label{LLN appendix}

We prove parts 1 and 2 in the next two subsections.

\subsection{Proof of Proposition \ref{BD_LLN_prop}}
When $A\geq 0$, from the current-tagged particle relation, we have
$
\{X_{Nt}\geq AN\} = \big\{ \frac{1}{N}\sum_{z\leq AN} \eta_{Nt}(z) \leq \frac{1}{N}\sum_{z\leq -1} \eta_0(z)\big\}$.
Since $\pi^N_t$ converges to $\rho(t, u)du$ in probability, we have
$$\lim_{N\rightarrow\infty}\P(X_{Nt}\geq AN) = \one\Big(\int_{-\infty}^A \rho(t, u)du \leq \int_{-\infty}^0\rho_0(u)du\Big).$$
When there is a
unique value $A=v_t\geq 0$ such that $\int_{-\infty}^A\rho(t,u)du = \int_{-\infty}^0\rho_0(u)du$, we conclude
$X_{Nt}/N$ converges to $v_t$ in probability.  

Similarly, when the unique value $A=v_t<0$, we use the relation $\{X_{Nt} < AN\} = \big\{\frac{1}{N}\sum_{z\leq AN} \eta_{Nt}(z) \leq \frac{1}{N}\sum_{z\leq -1}\eta_0(z)\big\}$ to deduce the convergence of $X_{Nt}/N$ to $v_t$.  \qed

\subsection{Proof of Proposition \ref{BD_LLN_prop1}}

We begin with a tightness result.

\begin{lem}
\label{tightness lemma}
Consider ASEP starting from a configuration $\eta_0$.  Then, all limit points in $D([0,1], \R)$ of the scaled positions $\{X_{Nt}/N\}_{t\in [0,1]}$ have continuous paths.  
\end{lem}

\begin{proof}
By the development in Section \ref{martingale section},
we have
$$X_t - \int_0^t \left\{p\left(1-\eta(X_r+1)\right) - q\left(1-\eta(X_r-1)\right) \right\}dr = M_1(t),$$
is a martingale whose quadratic variation satisfies
$E[M_1^2(t)] = O(t)$.

Then, one computes that
$$\frac{1}{N^2}E[(X_{Nt} - X_{Ns})^2] \leq C|t-s|^2 + \frac{C}{N}|t-s|.$$
Hence, by standard arguments,
 all limit points of $X_{N\cdot}/N$ in $D([0,1], \R)$ have continuous paths in $D([0,1], \R)$.
\end{proof}

\begin{rem}
\label{rem-A1}\rm
Then, to identify the uniform limit of $X_{N\cdot}/N$ as $v_\cdot \in C([0,1])$, that is $\lim_{N\rightarrow\infty}\sup_{t\in [0,1]}|X_{Nt}/N - v_t| =0$, it would be enough to establish for almost all $t\in [0,1]$ that
$X_{Nt}/N \rightarrow v_t$
in probability as $N\rightarrow \infty$.
\end{rem}

\subsection{Conclusion of proof of Part 1}
\label{first proof section}
By Remark \ref{rem-A1}, to complete the proof it is sufficient to identify $v_t\equiv \gamma t$. Although such an identification follows from the one time $t$-LLN implicit in the fluctuation result in \cite{TW}, we supply a different, self-contained proof of this identification for the interested reader.

Consider a labeling of the particles in $\eta_{step}$ going backwards: The tagged particle is the $0$th particle, the ones behind it are the $1$st, $2$nd, etc. particles.   We now view the system in terms of the gaps between particles, a zero-range process $\zeta_t$ with an infinite reservoir at site $0$, where
$\zeta_t(i)$ is the gap between the $i$th and $i-1$st particles for $i\geq 1$.  The generator is 
\begin{align*}
	\mathcal{L}_{ZR}f(\zeta) 
	&=  p \left[ f(\zeta^{0,1}) - f(\zeta) \right] \\
&\ \ \ \ \ \ 	+ \sum_{x\ge1}\one({\{\zeta(x) \ge 1\}}) \Big\{  p\left[f(\zeta^{x,x+1}) - f(\zeta) \right]  + q\left[ f(\zeta^{x,x-1}) - f(\zeta) \right] \Big\},
\end{align*}
where $\zeta^{x,y}$ for $x,y\geq 1$ corresponds to the update of a zero-range particle at $x$ moving to $y$, and $\eta^{0,1}$, $\eta^{1,0}$ are updates which adds to or takes away a zero-range particle from location $1$.
See also \cite{CS} which considers tagged particle asymptotics when $\gamma<0$ using the zero-range mapping, as well as anomalous behaviors when $\gamma=0$.

Since $p>q$, there are more particles entering than leaving the system.  We expect the zero-range system to be transient.  Consider the compactification by including the point `$\infty$' as a value for the zero-range particle number at a site.  In this sense, the system as a Markov process on the compact space $\big[\{0,1,2,\ldots\}\cup \{\infty\}\big]^\N$, where if a coordinate $\zeta_s(x)=\infty$ at time $s$, then at later times, the coordinate remains at $\infty$.  Since the state space is compact, this extended chain possesses invariant measures. 

Any such invariant measure $\mu$ must satisfy, for $k\geq 2$,
\begin{align*}
0&= E_\mu \mathcal{L}_{ZR}\zeta(1) = p + 	q\mu(\zeta(2)\geq 1) - \mu(\zeta(1)\geq 1)\\
0& = E_\mu\mathcal{L}_{ZR}\zeta(k) = p\mu(\zeta(k-1)\geq 1) + q\mu(\zeta(k+1)\geq 1) - \mu(\zeta(k)\geq 1).
\end{align*}
Let $\mu(\zeta(k)\geq 1) = \alpha_k$ for $k\geq 1$.  Then, $p(1-\alpha_1) = q(\alpha_1-\alpha_2)$ and $p(\alpha_{k-1} -\alpha_k) = q(\alpha_k-\alpha_{k+1})$ for $k\geq 2$.  Hence,
$\alpha_k -\alpha_{k+1} = \big({p}/{q}\big)^k(1-\alpha_1)$
for $k\geq 1$.  Since $p>q$ and $\max_k\alpha_k\leq 1$, by taking the limit as $k\uparrow\infty$ we have $\alpha_1=1$, which implies that $\alpha_k=1$ for $k\geq 1$.  In particular, any invariant measure $\mu$ is such that $\mu(\zeta(1)=0)=0$.

By considering subsequences of the distributions $t^{-1}\int_0^t \P_{\zeta_0}(\zeta_s\in \cdot)ds$, we have
\begin{align}
\label{step invariant}
\limsup_{t\uparrow\infty} \frac{1}{t}\int_0^t \P_{\zeta_0}(\zeta_s(1)=0)ds \leq \max_\mu\mu(\zeta(1)=0) =0.
\end{align}
 Consider now, starting from the step profile, that
$M_1(t) = X_t - \int_0^t \gamma +q\eta_s(X_s -1)ds$
is a martingale with quadratic variation $\E_{\eta_{step}}[M_1^2(t)] = \E_{\eta_{step}}\big[\int_0^t 1 - q\eta_s(X_s -1)ds\big] \leq t$ (cf. Section \ref{martingale section}).
We see then that 
\begin{align}
\label{ldp 1}
&\lim_{t\rightarrow\infty}\frac{1}{t}X_t = \gamma \ \ {\rm in \ probability\ } \  \Longleftrightarrow \ \lim_{t\rightarrow\infty}\frac{1}{t}\E_{\eta_{step}}\left[\int_0^t \eta_s(X_s-1)ds\right] = 0.
\end{align}
The meaning of the last limit is that the tagged particle does not meet often the particle behind it with respect to law of large numbers scaling.
Hence, with respect to the exclusion particles-zero-range gaps mapping, as $\eta_s(X_s -1) = \one(\zeta_s(1)=0)$, we conclude by \eqref{step invariant} and \eqref{ldp 1} that
$\lim_{N\rightarrow\infty}X_{Nt}/N =\gamma t$ in probability. \qed

\subsection{Conclusion of proof of Part 2}
By Remark \ref{rem-A1}, we need only identify $v_t\equiv \gamma t$.  We will 
make use of the `shock fluctuation' coupling results in \cite{Ferrari}; see also \cite{FFV} in this context with respect to TASEP.  

\smallskip
{\it Step 1.} One of the results in \cite{Ferrari} considers a layer of say basement class particles initially distributed according to $\prod_{x\in \Z}{\rm Bern}(a)$.  On top of this layer, we put say class $0^*$ particles initially distributed according to $\prod_{x<0}{\rm Bern}(\alpha)\times \prod_{x\geq 0}{\rm Bern}(0)$ and say class $0^\dag$ particles initially distributed according to $\prod_{x<0}{\rm Bern}(0)\times\prod_{x\geq 0}{\rm Bern}(\alpha)$.

In the ASEP evolution, the class $0^*$ particles are set to be `second-class' to the class $0^\dag$ particles, which in turn are taken `second-class' to the basement class particles, more formally stated through the basic coupling.  At time $t$, let $Z^0_t$ be the position of the right-most class $0^*$ particle, and let $W^0_t$ be that of the left-most class $0^\dag$ particle.  Then, in \cite{Ferrari}, it is shown that both $Z^0_{Nt}/N$ and $W^0_{Nt}/N$ limit in probability to $\gamma(1-\alpha-a)t$.  We will call this the `1-layer' phenomena.  We comment, if the density $a=0$, then $W^0_t$ is in fact the position of the left-most `first-class' or tagged particle starting under distribution $\prod_{x<0}{\rm Bern}(0)\times\prod_{x\geq 0}{\rm Bern}(\alpha)$.

One may stack another layer on top of the class $0$ layer:  We may distribute initially class $1^*$ particles according to $\prod_{x< b N}{\rm Bern}(\beta)\times \prod_{x\geq bN}{\rm Bern}(0)$ and class $1^\dag$ particles according to $\prod_{x<bN}{\rm Bern}(0)\times\prod_{x\geq b N}{\rm Bern}(\beta)$.  The value $b$ will be called a `shock' location.  Then, in the ASEP motion, we set the class $1^*$ particles as `second-class' to the class $1^\dag$ ones, which we set as `second-class' to both class $0$, as well as basement class particles.  Let $Z^2_t$ be the right-most position of the class $1^*$ particles, and $W^2_t$ be the left-most position of the class $1^\dag$ particles at time $t$.  Since the class $1$ particles are `second-class' to the class $0$ and basement class ones, considering together the class $0$ and the basement class particles as one layer, the `1-layer' phenomena in \cite{Ferrari} gives both $Z^1_{Nt}/N$ and $W^1_{Nt}/N$ converge to $b + \gamma(1-(\alpha+a)-(\alpha+a+\beta))t$.  These veloicites correspond to the Rankine-Hugoniot speed of the shock $\gamma(R(1-R) - L(1-L))/(R-L) = 1-L-R$ with respect to a Riemann step initial condition with left and right densities $L$ and $R$ respectively. 

In this way, one may stack several layers, and deduce LLNs for the right-most and left-most particles in each layer, subordinate to particles in layers beneath it.

\vskip .1cm
{\it Step 2.}  Consider now the position $X_t$ of the tagged particle.  By coupling, it is less than the position of a particle in a system with no other particles, and therefore $\limsup_{N\uparrow\infty}X_{Nt}/N$ is less than $\gamma t$ for $t\geq 0$.  

To show the reverse inequality, consider a `steps' profile:
\begin{align}
\label{steps profile}
\rho_1(u) = \left\{\begin{array}{ll}
0& {\rm   for \ }u\leq 0\\
\epsilon& {\rm   for \ }0<u\leq 2\gamma \epsilon\\
\epsilon + \delta_1& {\rm   for \ }2\gamma\epsilon <u\leq 2\gamma(\epsilon+\delta_1)\\
\vdots&\vdots\\
\epsilon + \sum_{r=1}^{k}\delta_r & {\rm   for \ }2\gamma(\epsilon + \sum_{r=1}^{k-1}\delta_{r})<u\leq 2\gamma \rho\\
\rho& {\rm   for \ }u>2\gamma\rho.
\end{array}\right.
\end{align}
Here, $\epsilon<\rho$, $\delta_r = \kappa \delta_{r-1}$ for $2\leq r\leq k$ and $0<\kappa<1$.  We may take the value of $\delta_1 \leq \epsilon/2$, and $\kappa$ and $k$ 
so that $\epsilon + \delta_1\sum_{r=0}^{k-1} \kappa^r = \rho$.

Such an initial profile $\rho_1$ is larger than $\rho_0$, given in the statement of Proposition \ref{BD_LLN_prop1}.  Hence, by comparison coupling, the tagged motion $Y_t$, starting under $\nu_{\rho_1(\cdot)}(\cdot|\eta(0)=1)$, at time $t$ will be less than $X_t$.  We will show that $\lim_{N\rightarrow\infty} Y_{Nt}/N = \gamma(1-\epsilon)t$.  As $\epsilon>0$ is arbitrary, this will finish the proof.

 The idea is that the LLN behavior of $Y_t$ is the same as that of the tagged particle starting under distribution $\prod_{x<0}{\rm Bern}(0)\times \prod_{x\geq 0}{\rm Bern}(\epsilon)$, for which a LLN is known by the discussion in Step 1.

\vskip .1cm
{\it Step 3.}   Consider $k+1$ layers of `second-class' particles, as in Step 1:  Taking the density of the basement class layer to be $a=0$, let the bottom class $0$ (layer $0$) density correspond to $\alpha = \epsilon$.  Let the next layer $1$ correspond to $\beta = \delta_1$ and $b = 2\gamma\epsilon$.  Then, put layer $r$ on top of layer $r-1$ with density $\delta_r$ and initial macroscopic `shock' location $2\gamma(\epsilon + \sum_{j=1}^{r-1}\delta_j)$ for $2\leq r\leq k$.  

In layer $0\leq r\leq k$, for each $0\leq s\leq 1$, the right-most $r^*$ and left-most $r^\dag$ particle scaled positions $Z^r_{Ns}/N$ and $W^r_{Ns}/N$ converge in probability to
\begin{align}
\label{velocities}
\begin{array}{ll}
\gamma(1-\epsilon)s & {\rm   for \ }r=0\\
2\gamma\epsilon + \gamma(1-2\epsilon - \delta_1)s & {\rm   for \ }r=1\\
2\gamma(\epsilon + \sum_{j=1}^{r-1} \delta_j) + \gamma(1-2\epsilon-2\sum_{j=1}^{r-1}\delta_j-\delta_r)s& {\rm   for \ }2\leq r\leq k,
\end{array}
\end{align}
noting the left and right densities $L$ and $R$ given in \eqref{steps profile} over the layers.
For each $0\leq s\leq 1$, these positions are strictly increasing.  Also, at time $s=1$, these scaled positions reduce to $\gamma(1-\delta_r)$ for $0\leq r\leq k$, if we call $\delta_0 = \epsilon$.

By the argument of Lemma \ref{tightness lemma}, since compensators of the jumps of $Z^r_{Ns}$ and $W^r_{Ns}$ are uniformly bounded $O(N)$, one can conclude $Z^r_{Ns}/N$ and $W^r_{Ns}/N$ approximate the positions \eqref{velocities} 
uniformly in $0\leq s\leq 1$ for all large $N$, with high probability.

Hence, since the velocities are strictly ordered over the layers, $\dag$ particles to the right of positions $\{W^r_{Ns}\}_{s\in [0,1]}$ do not interact with the higher order $*$ particles to the left of $\{Z^q_{Ns}\}_{s\in [0,1]}$ for $q\leq r-1$, with high probability.  Also, by construction, $*$ particles in layer $r$ do not interact with $\dag$ ones in layers $q\geq r$.

Therefore, the scaled limits of $\{W^r_{Ns}/N\}_{s\in [0,1]}$, over $0\leq r\leq k$, are the same as if there were no $*$ `second-class' particles in any layer in the system.

\vskip .1cm
{\it Step 4.}  But, we may consider the system starting from the `steps' profile \eqref{steps profile}.  We may view the evolution of it in terms of the basic coupling where there is a bottom layer initially distributed according to $\prod_{x<0}{\rm Bern}(0)\times\prod_{x\geq 0}{\rm Bern}(\epsilon)$, a next layer subordinate to it distributed according to $\prod_{x<2\gamma\epsilon N}{\rm Bern}(0)\times\prod_{x\geq 2\gamma \epsilon N} {\rm Bern}(\delta_1)$, and so on to a final layer subordinate to the layers below with distribution $\prod_{x<2\gamma(\epsilon + \delta_1+\cdots + \delta_{k-1})N}{\rm Bern}(0)\times \prod_{x\geq 2\gamma(\epsilon + \delta_1+\cdots+\delta_{k-1})N} {\rm Bern}(\delta_{k})$.

The left-most second-class particles in these layers behave as $\{W^r_{Ns}\}_{s\in [0,1]}$ in the LLN scale from the discussion in Step 3.  Moreover, the bottom left-most particle $\{W^0_{Ns}\}_{s\in [0,1]}$ in the system does not see the second-class particles above it, with high probability.  Hence, $\{Y_{Ns}\}_{s\in [0,1]}$, the left-most particle in the system without any second-class particles, has the same LLN behavior as $\{W^0_{Ns}\}_{s\in [0,1]}$, which has velocity $\gamma(1-\epsilon)$ as desired.
\qed

\section{Proof of Proposition \ref{BD_LDP_prop}: Large deviation bound from step profile}
\label{LDP section}

Consider ASEP, with $0<\gamma \leq 1$, starting from the step initial profile $\eta_{step}$.  We will view it in terms of the associated zero-range system of gaps between particles discussed in Section \ref{first proof section}.
Since the zero-range rate $g(k) = \one(k\geq 1)$ is increasing, the system is `attractive' in that it satisfies the basic coupling.  Hence, starting from a FKG measure, such as the product of point masses consisting of the initial step condition, the system has positive correlations; see Liggett \cite{Liggett} for a discussion of the basic coupling, FKG measures and positive correlation.

The position $X_t$ of the tagged particle is the total sum $\sum_{j\geq 1} \zeta_t(j)$ of zero-range particles at time $t$, in other words the integrated current across the reservoir.  Here, the variable $1-\eta_s(X_s-1) = \one(\zeta_s(1)\geq 1)$.  
The last two functions of $\zeta$ are increasing and therefore are positively correlated.  In particular, the functions $\eta_s(X_s-1)= \one(\zeta_s(1)=0)$ and $\sum_{i\geq 1}\zeta_t(i)$, as the first is decreasing and the second is increasing, are negatively correlated.

By Proposition \ref{BD_LLN_prop1}, one has $X_t/t \rightarrow \gamma$ in probability, and so, from negative correlations and \eqref{ldp 1}, we conclude
\begin{align}
\label{neg corr relation}
\frac{1}{t}\E_{\eta_{step}}\Big[\int_0^t \one(\eta_s(X_s -1)=1) ds\big |X_t>At\Big]  \leq \frac{1}{t}\E_{\eta_{step}}\Big[\int_0^t \one(\eta_s(X_s-1)=1) ds\Big] \rightarrow 0. 
\end{align}

Consider the exponential martingale (cf. Section \ref{martingale section}), for $\lambda \in \R$,
$$M_2(t)=\exp\left\{\lambda X_t - tp(e^\lambda -1) -q(e^{-\lambda}-1)\int_0^t (1-\eta_s(X_s-1))ds\right\}.$$
Write
\begin{align*}
\frac{1}{t}\log \P_{\eta_{step}}(X_t>At) & = \frac{1}{t}\log \P_{\eta_{step}}\left(\int_0^t \eta_s(X_s -1)ds \leq \epsilon t, X_t>At\right)\\
& \ \ + \frac{1}{t}\log\left[ 1 +\frac{\P_{\eta_{step}}(\int_0^t \eta_s(X_s -1)ds > \epsilon t, X_t>At)}{\P_{\eta_{step}}(\int_0^t \eta_s(X_s -1)ds \leq \epsilon t, X_t>At)}\right].
\end{align*}
The second term vanishes according to \eqref{neg corr relation} for each $\epsilon>0$.  However, analysis of the first term, using Chebychev's inequality, the exponential martingale, optimization over $\lambda>0$, and taking $\epsilon\downarrow 0$,
yields a `birth-death' process bound:  Indeed, we arrive at the bound $\sup_{\lambda>0}\left\{-A + p(e^\lambda -1) + q(e^{-\lambda} -1)\right\}$.  When $A>\gamma$, one can solve that the supremum is achieved at a unique $\lambda>0$ where $pe^\lambda -qe^{-\lambda} = A$.  Hence, with $c=e^\lambda$, we have
$$\limsup_{t\rightarrow\infty}\frac{1}{t}\log \P_{\eta_{step}}(X_t>At) \leq 
-A\log c + pc +q/c -1= -\I^Z(A)$$
for $c=\big(A + \sqrt{A^2 + 4pq}\big)/2p$.  For $A=\gamma$, as $-\I^Z(A)=0$, the bound holds trivially.  \qed

 \section{Proof of Lemma \ref{superexponential cost lemma}: Superexponential bound for the flux}
\label{section superexponential}

 Recall $\rho_0(u)=0$ for $|u|>R$ with $R$ large and, in term of the entropic solution of \eqref{Burgers}, 
$v(t,u) = \int_{-\infty}^u \rho(t, z)dz$.  
The Hopf-Lax formula \eqref{Hopf-Lax formula} for $v(t,u)$ with respect to $G$ (cf. \eqref{G-def-intro}) is written
 \begin{align}
 \label{U_0 eq}
 v(t,u) = \sup_y \left\{v(0,y) - tG\left(\frac{y-u}{t}\right)\right\} = v(0, b(t,u)) - tG\left(\frac{b(t,u)-u}{t}\right),
 \end{align}
 in terms of an argmax $y=b(t,u)$.  Indeed, $v(0,y) - tG\left(\frac{y-u}{t}\right) = v(0,y)-(y-u)$, for $y-u>t\gamma$, decreases as $y$ increases.  Also, $v(0,y) - tG\left(\frac{y-u}{t}\right) = v(0,y)$, for $y-u< -t\gamma$, decreases as $y$ decreases.  Hence, a maximum $y=b(t,u)$ is achieved.
 
 Suppose the initial condition $\eta_0$ is such that $\eta_0(x)=0$ for $|x|>RN$, and $\pi^N_0\in B_\delta(\rho_0)$ for $\delta>0$, as specified in Lemma \ref{superexponential cost lemma}.
 Consider, with respect to the exclusion generator $N\mathcal{L}$ (cf. \eqref{exclusion gen}) for the speeded up process, and a test function $J=J(t,u)$, the stochastic differential and martingale $M(t)$, 
 \begin{align}
 \label{M_J}
 \frac{1}{N}\sum_{x\in \Z} J(t,\frac{x}{N})\eta_{Nt}(x) - \frac{1}{N}\sum_{x\in \Z} J(0,\frac{x}{N})\eta_0(x)
 = \int_0^t \frac{1}{N}\sum_{x\in \Z} \partial_t J(s, \frac{x}{N})\eta_{Ns}(x) ds&\\
\ \  + \frac{1}{N}\sum_{\pm}\int_0^t N p(\pm 1)\sum_{x\in \Z} \big(J(s,\frac{x\pm 1}{N})-J(s,\frac{x}{N})\big)\eta_{Ns}(x)(1-\eta_{Ns}(x\pm 1)) ds + M(t).&\nonumber
 \end{align}

Define the function $F_0$ by
$$F_0(u)=\left\{\begin{array}{ll}
1& {\rm for \ }u\leq 0\\
0&{\rm for \ }u\geq \delta\\
\delta^{-1}(u-\delta) &{\rm for \ } 0\leq u \leq \delta.
\end{array}\right.
$$
Let $F$ be the mollification of $F_0$ by a smooth function with support in $[-\delta/10, \delta/10]$, say.
We now take the test function $J$, with respect to $L$ and $b=b(t,L)$, in the form
$$J(s,u) = F\left(u-b - \frac{s}{t}(L-b)\right).$$
Such a test function is allowed as $\eta_0(x)= 0$ for $|x|>RN$.
Then,
$J(t,u) = F(u-L) \ \ {\rm and \ \ } J(0,u) = F(u-b)$.

Substituting into \eqref{M_J}, we have
\begin{align*}
&\frac{1}{N} \sum_{x\in \Z} F\left(\frac{x}{N} - L\right)\eta_{Nt}(x) - \frac{1}{N}\sum_{x\in \Z} F\left(\frac{x}{N}-b\right)\eta_0(x)\\
&= \int_0^t \frac{1}{N}\sum_{x\in \Z} F'\left(\frac{x}{N}-b - \frac{s}{t}(L-b)\right)\\
&\ \ \ \ \ \ \ \ \ \ \ \times \Big[ -\frac{L-b}{t}\eta_{Ns}(x) +\sum_{\pm} \pm p(\pm 1)\eta_{Ns}(x)(1-\eta_{Ns}(x\pm 1))\Big]ds \\
&\ \ \ \ + D_1(t) + M(t)\\
&= \int_0^t \frac{1}{N}\sum_{x\in \Z} F'\left(\frac{x}{N}-b-\frac{s}{t}(L-b)\right)\left[-\frac{L-b}{t}\eta^{k}_{Ns}(x) + \gamma\eta^k_{Ns}(x)(1-\eta^k_{Ns}(x))\right] ds\\
&\ \ \ \  + D_2(t) + D_1(t) + M(t).
\end{align*}
Here,
\begin{align*}
D_1(t) &= \sum_{\pm}\int_0^t \frac{p(\pm 1)}{N}\sum_{x\in \Z} N\left(J(s,\frac{x\pm 1}{N}) -J(s,\frac{x}{N})\mp \partial_u J(s,\frac{x}{N})\right)\\
&\ \ \ \ \ \ \ \ \ \ \ \ \ \ \ \ \ \ \times \eta_{Ns}(x)(1-\eta_{Ns}(x\pm 1))ds,\\
\eta^k_{Ns}(x) &= \frac{1}{2k+1}\sum_{z: |z-x|\leq k} \eta_{Ns}(z)
\ \ {\rm and} \\
D_2(t) &= \int_0^t \Big\{\frac{1}{N}\sum_{x\in \Z} F'\left(\frac{x}{N}-b - \frac{s}{t}(L-b)\right)\left[\frac{L-b}{t}\big(\eta^k_{Ns}(x) -  \eta_{Ns}(x)\big) \right]\\
&\ \ \ \ +  \sum_\pm \pm p(\pm 1)\left[\eta_{Ns}(x)(1-\eta_{Ns}(x\pm 1)) - \eta^k_{Ns}(x)(1-\eta^k_{Ns}(x))\right]\Big\} ds.
\end{align*}

 Now, by \eqref{hopf 1},
 $
 \sup_{0\leq w\leq 1} \left\{ - \frac{L-b}{t}w + \gamma w(1-w)\right\} 
 = G\left(\frac{b-L}{t}\right)$.
 Hence, as $F'\leq 0$,
 \begin{align*}
& \frac{1}{N}\sum_{x\in \Z} F\left(\frac{x}{N} -L\right)\eta_{Nt}(x) - \frac{1}{N}\sum_{x\in \Z} F\left(\frac{x}{N}-b\right)\eta_0(x)\\
&\geq G\left(\frac{b-L}{t}\right)\int_0^t \frac{1}{N} \sum_{x\in \Z} F'\left(\frac{x}{N} -b - \frac{s}{t}(L-b)\right) ds + D_2(t)+D_1(t)+M(t).
\end{align*}

Since
$F'\left(\frac{x}{N} - b - \frac{s}{t}(L-b)\right) \cdot \frac{b-L}{t} = \partial_s F\left(\frac{x}{N} -b - \frac{s}{t}(L-b)\right)$,
we have
\begin{align*}
\int_0^t \frac{1}{N}\sum_{x\in \Z} F'\left(\frac{x}{N} -b - \frac{s}{t}(L-b)\right)ds
= \frac{t}{N(b-L)} \sum_{x\in \Z} \left[F\left(\frac{x}{N}-L\right) - F\left(\frac{x}{N}-b\right)\right].
\end{align*}
Then,
   \begin{align}
   \label{super 1}
& \frac{1}{N}\sum_{x\in \Z} F\left(\frac{x}{N} -L\right)\eta_{Nt}(x) \\
 &\geq \frac{1}{N}\sum_{x\in \Z} F\left(\frac{x}{N}-b\right)\eta_0(x)-  tG\left(\frac{b-L}{t}\right)  + D_3(t)+ D_2(t)+D_1(t)+M(t)\nonumber
 \end{align}
 where
 $D_3(t) = tG\left(\frac{b-L}{t}\right)\left[ 1+\frac{1}{N(b-L)} \sum_{x\in \Z} \left[F\left(\frac{x}{N}-L\right) - F\left(\frac{x}{N}-b\right)\right]\right]$.
 
 Note that
 \begin{align}
\label{9-1}
\frac{1}{N}\sum_{x\in \Z} F\left(\frac{x}{N} -L \right)\eta_{Nt}(x) &= \frac{1}{N}\sum_{x\leq LN} \eta_{Nt}(x) + O\Big(\frac{1}{\delta N}\sum_{0\leq z\leq \delta N}\frac{z}{N}\Big)\nonumber\\
& = \frac{1}{N}\sum_{x\leq LN} \eta_{Nt}(x) + O(\delta).
\end{align}
Similarly, via the assumption $\pi^N_0\in B_\delta(\rho_0)$ on the initial condition,
 \begin{align}
\label{9-2}
\frac{1}{N}\sum_{x\in \Z} F\left(\frac{x}{N}-b\right)\eta_0(x) = \frac{1}{N}\sum_{x\leq bN}\eta_0(x) + O(\delta) = v(0,b) + O(\delta).
\end{align}

Therefore, from \eqref{U_0 eq}, \eqref{9-1}, \eqref{9-2}, plugging into \eqref{super 1}, we have
that
\begin{align}
\label{super_last}
\frac{1}{N}\sum_{x\leq LN} \eta_{Nt}(x) &\geq  v(0,b) - tG\left(\frac{b-L}{t}\right) + O(\delta) +D_4(t)\nonumber\\
&= v(t, L) + O(\delta) + D_4(t),
\end{align}
where $D_4(t) = \sum_{i=1}^3 D_i(t) + M(t)$.
 
 We now bound the errors and therefore $D_4(t)$.  Uniformly over $\eta_\cdot$: (1) By smoothness of $F$, $|D_1(t)| = O(1/N)$, and (2) by calculations as in \eqref{9-1}, \eqref{9-2}, $|D_3(t)| = O(\delta)$.  Then,
 $\lim_{\delta\downarrow 0}\lim_{N\rightarrow\infty}\sup_{\eta_\cdot} |D_1(t) + D_3(t)| =0$.
 
  A uniform superexponential bound for the `$1$-block' is well known: 
  
  \noindent 
 $\lim_{k\uparrow\infty}\lim_{N\rightarrow\infty}\frac{1}{N}\log  \sup_{\eta_0}\P_{\eta_0}(|D_2(t)|>\delta) = -\infty$, where the supremum is over $\eta_0$ supported on $[-RN, RN]$.  Indeed, the Radon-Nikodym derivative in the change of measure from $\eta_0$ to $\nu_{\rho_0(\cdot)}$ is uniformly bounded of order $e^{CN}$ where $C=C_{R, \rho_0(\cdot)}$ is a constant.  One now follows the standard route in \cite{kl}[Chapter 10]; see also \cite{Jensen_thesis}[Lemma 2.4]. 
 
Also, a uniform martingale bound is known:
 \begin{align}
 \label{unif mart}
 \lim_{N\rightarrow\infty} \frac{1}{N}\log \sup_{\eta_0}\P_{\eta_0}\big(|M(t)|>\delta \big) = -\infty.
 \end{align}
 Indeed, we adapt the argument in \cite{Jensen_thesis}[Lemma 2.2] to ASEP, which is uniform over $\eta_0$ supported on $[-RN, RN]$: Note that 
$Z_t^N=\exp\big\{f(t, \eta_{Nt}) - f(0, \eta_0) - \int_0^t e^{-f}\big(\partial_s + N\mathcal{L}\big)e^fds\big\}$ is a martingale, and 
$e^{-f}\big(\partial_s + N\mathcal{L}\big)e^f = (\partial_s + N\mathcal{L})f + O\big(R\|J'(s, \cdot)\|^2_{L^\infty}\big)$
 when $f(t, \eta) = \sum_x J(t, x/N)\eta(x)$.  Hence, $Z^N_t = \exp\{NM(t) +O\big(R\|J'(s, \cdot)\|^2_{L^\infty}\big)\}$ and for $\lambda>0$, $E_{\eta_0}\big[\exp\big\{\pm N\lambda M(t)\big\}\big] \leq e^{C_{R,J} \lambda^2}$, as $\pm \lambda M(t)$ corresponds to $J$ being replaced by $\pm \lambda J$. By Chebychev's inequality, $P_{\eta_0}(\pm M(t)>\delta) \leq e^{C_{R,J}\lambda^2} e^{-N\lambda\delta}$, and \eqref{unif mart} follows by taking $N\uparrow\infty$ and then $\lambda\uparrow\infty$.

Hence, via the bound of $D_4(t)$, and inequality \eqref{super_last}, with $\delta = \delta(\varepsilon)>0$ small enough, the left side of \eqref{super exp statement}
is less than $\lim_{N\rightarrow\infty} \frac{1}{N}\log \sup_{\eta_0}\P_{\eta_0}\left(-\varepsilon \geq O(\delta)+ D_4(t)\right) = -\infty$.
\qed

\appendix
 
\section{Hopf-Lax variational formulation}
\label{Hopf Lax section}
We now state a Hopf-Lax variational form associated to the entropic solution of \eqref{Burgers} (cf. \cite{Evans}[Section 3.3]).  Suppose $\rho_0(u) = 0$ for $|u|> \z$, some $\z$ large.  
 For $t>0$, let
$$v(t,u) = \int_{-\infty}^u \rho(t, z)dz$$
and note that
$\partial_t v + \gamma(\partial_u v) (1-\partial_u v) = 0$.

Form $w(t,u) = u-v(t,u)$ and note  
\begin{align*}
\partial_t w = -\partial_t v
 = \gamma(\partial_u v)(1-\partial_u v) = \gamma(1-\partial_u w)(\partial_u w).
 \end{align*} 
 Since $0\leq \partial_u w = 1-\rho(t,u)\leq 1$, in terms of the extended convex function, 
 $$L(r) = \left\{\begin{array}{ll}
 -\gamma r(1-r)& {\rm for \ }0\leq r\leq 1\\
 \infty&{\rm for \ }r<0 {\rm \ or \ } r>1,
 \end{array}\right.
 $$
 we have
$\partial_t w + L(\partial_u w) = 0$.

Observe that the convex conjugate of $L(r)$ is calculated as
\begin{align}
\label{hopf 1}
G(z) &= \sup_{0\leq y\leq 1}\big\{yz - L(y)\big\} =\sup_y\big\{yz - L(y)\big\} \\
&= \left\{\begin{array}{ll}
\frac{\gamma}{4}\big(1+\frac{z}{\gamma}\big)^2& {\rm for \ }|z|\leq \gamma\\
0& {\rm for \ }z<-\gamma\\
z&{\rm for \ }z>\gamma.
\end{array}\right. \nonumber
\end{align}
Moreover, for $0\leq r \leq 1$, we have
$
L(r)=
 \sup_{-\gamma \leq y\leq \gamma} \big\{r y - G(y)\big\}
 = \sup_y \big\{r y - G(y)\big\}
$.

Then, by the Hopf-Lax formula,
$$w(t,u) = \inf_y \left\{ t G\left(\frac{u-y}{t}\right) + w(0, y)\right\}$$
and hence
$$u-v(t,u) = \inf_y \left\{tG\left(\frac{u-y}{t}\right) + y - v(0,y)\right\}.$$
Let $z=u-y$.  When $|z/t|\leq \gamma$,
$$z - tG\left(\frac{z}{t}\right) = \frac{t\gamma z}{t\gamma} - \frac{ t\gamma}{4}\left(1 +  \frac{z}{t\gamma} \right)^2 = \frac{-t\gamma}{4}\left(1- \frac{z}{t\gamma}\right)^2 = -tG\left(-\frac{z}{t}\right).$$
Whereas, when $z/t>\gamma$, we have $z-tG(z/t) = -tG(-z/t)=0$, and when $z/t<-\gamma$ we have $z-tG(z/t) = -tG(-z/t) = z$.
Therefore, we have
\begin{align}
\label{Hopf-Lax formula}
v(t,u) &= \sup_y \left\{ v(0,y) + u-y - tG\left(\frac{u-y}{t}\right)\right\} \ = \ \sup_y \left\{v(0,y) - tG\left(\frac{y-u}{t}\right)\right\}.
\end{align}

\section{Poisson process law of tagged particle in TASEP}
 \label{Poisson process section}

  There are at least two proofs of the following result.  A version by `time-reversal', given below for convenience of the reader; see \cite{Kipnis}. 
	Another, using generator calculations, can be found in \cite{Liggett}[Theorem 4.7].

 \begin{prop}
 Consider TASEP.  Starting under $\nu_\rho(\cdot|\eta(0)=1)$, we have that $\{X_t:t\geq 0\}$ has the distribution of a Poisson process with rate $1-\rho$.
 \end{prop}

	\begin{proof}
 The argument considers the zero-range dynamics of the spacings $\zeta_t$ between particles in TASEP.  The setting is similar to the zero-range model considered in Section \ref{first proof section}.  
We can assume there are no particles behind the tagged particle as these wouldn't interact with the tagged particle in TASEP.  Label the particles going forward:  The tagged particle is the $0$th particle, the next to the right is particle $1$, etc.  In this setting, let $\zeta(i)$ be the gap between particles $i$ and $i-1$ for $i\geq 1$.  Then, we set $\zeta(0)=\infty$, a `reservoir' behind the tagged particle.  The position at time $t$ of the tagged particle in TASEP is the number $J(t)$ of zero-range particles which jump from location $1$ to the reservoir $0$ in the zero-range process.   When TASEP starts in the invariant state $\prod_{i\leq -1}{\rm Bern}(0)\times{\rm Bern}(1)\times \prod_{i\geq 2}{\rm Bern}(\rho)$, the zero-range process $\{\zeta_\cdot(j):j\geq 1\}$ is also in an invariant state, $\nu_{ZR}=\prod_{i\geq 1} {\rm Geometric}(1-\rho)$.  
 
The process generator is
$\mathcal{L}_{ZR}f(\zeta) = (f(\zeta^{1,0})-f(\zeta))\one(\zeta(1)\geq 1) + \sum_{j\geq 2} (f(\zeta^{j, j-1})-f(\zeta))\one(\zeta(j)\geq 1)$
where $\zeta^{1,0}$ and $\zeta^{0,1}$ are the configurations which takes away and puts a zero-range particle at location $1$.
The $L^2(\nu_{ZR})$ adjoint 
may be found by computing on local functions as
$\mathcal{L}_{ZR}^*f(\eta) = (f(\zeta^{0,1})-f(\zeta))(1-\rho) + \sum_{j\geq 2} (f(\zeta^{j-1, j})-f(\zeta))\one(\zeta(j-1)\geq 1)$.
Since the adjoint corresponds to time-reversal of the process, $J(t)$ has the same distribution as $J^*(t)$, the number of zero-range particles which jump from the reservoir $0$ to $1$ in the reversed process.  One sees from the form of $\mathcal{L}_{ZR}^*$ that $J^*(t)$ is a Poisson process with rate $1-\rho$. \end{proof}

\vskip .2cm

{\bf Acknowledgements.}
This work was supported in part by a Simons Sabbatical grant, and ARO-W911NF-18-1-0311.

\end{document}